\documentclass[11pt]{article}
\usepackage[utf8]{inputenc}
\usepackage{amsmath, amssymb, amsfonts}
\usepackage{graphicx}
\usepackage{amsthm}
\usepackage{enumitem}
\usepackage{bm}
\usepackage{color}
\usepackage{geometry}
\theoremstyle{plain}
\newtheorem{thm}{Theorem}[section]

\newtheorem{clm}[thm]{Claim}

\theoremstyle{definition}

\newtheorem{ex}{Example}

\makeatletter
\@addtoreset{equation}{section}

\makeatother

\begin{document}

\title{Supermodular Extension of Vizing's Edge-Coloring Theorem}

\author{Ryuhei MIZUTANI\\
Department of Mathematical Informatics,\\
Graduate School of Information Science and Technology,\\
The University of Tokyo, Tokyo, 113-8656, Japan.\\
\texttt{ryuhei\_mizutani@mist.i.u-tokyo.ac.jp}
}

\maketitle

\begin{abstract}
K\H{o}nig's edge-coloring theorem for bipartite graphs and Vizing's edge-coloring theorem for general graphs are celebrated results in graph theory and combinatorial optimization. Schrijver generalized K\H{o}nig's theorem to a framework defined with a pair of intersecting supermodular functions. The result is called the supermodular coloring theorem.

This paper presents a common generalization of Vizing's theorem and a weaker version of the supermodular coloring theorem. To describe this theorem, we introduce intersecting 2/3-supermodular functions, which are extensions of intersecting supermodular functions. The paper also provides an alternative proof of Gupta's edge-coloring theorem using a special case of this supermodular version of Vizing's theorem.
\end{abstract}

\section{Introduction}
\subsection{Edge-coloring}
Let $G=(V,E)$ be a multigraph. An \textit{edge-coloring} of $G$ is an assignment of colors to all edges in $E$ such that no adjacent edges have the same color. The \textit{chromatic index} $\chi'(G)$ of $G$ is the minimum number $k$ such that there exists an edge-coloring of $G$ using $k$ colors. The \textit{degree} of a vertex $v\in V$ is the number of edges incident to $v$. K\H{o}nig \cite{konig1916} showed the following relation between the chromatic index $\chi'(G)$ and the maximum degree $\Delta(G)$ of a bipartite multigraph $G$. 
\begin{thm}[K\H{o}nig \cite{konig1916}]
\label{konig16}
$\chi'(G)=\Delta(G)$ holds for any bipartite multigraph $G$.
\end{thm}
\noindent It holds that $\chi'(G)\geq \Delta(G)$ for any multigraph $G$
because edges adjacent to the same vertex must have different colors. Theorem \ref{konig16} states that this lower bound $\Delta(G)$ is equal to $\chi'(G)$ for every bipartite multigraph. 

The \textit{multiplicity} $\mu(G)$ is the maximum number of edges between any pair of two vertices in $G$. Vizing \cite{vizing1965} showed the following analogue of Theorem \ref{konig16} for general multigraphs.

\begin{thm}[Vizing \cite{vizing1965}]
\label{vizing65}
$\Delta(G)\leq \chi'(G)\leq \Delta(G)+\mu(G)$ holds for any multigraph $G$.
\end{thm}

For a vertex $v\in V$, let $\delta(v)$ and $\mathrm{deg}(v)$ denote the set of edges incident to $v$ and its cardinality, respectively. For a positive integer $k$, we denote $[k]=\{1,2,\ldots, k\}$. For a color assignment $\pi:U\rightarrow [k]$ of a finite set $U$, we use the notation $\pi(X)=\{\pi(u)\mid u\in X\}$ for a subset $X\subseteq U$. Gupta \cite{gupta1974,gupta1978} generalized each of K\H{o}nig's theorem and Vizing's theorem to a framework including the packing problem of \textit{edge covers} (a set of edges such that every vertex is incident to at least one edge of the set). The following theorem is an extension of K\H{o}nig's theorem by Gupta \cite{gupta1978}:

\begin{thm}[Gupta \cite{gupta1978}]
\label{gupta78bipartite_general}
Let $G=(V,E)$ be a bipartite multigraph. For $k\in \mathbf{Z}_{>0}$, there exists a color assignment $\pi:E\rightarrow [k]$ such that $|\pi(\delta(v))|\geq \min\{\mathrm{deg}(v),k\}$ holds for every $v\in V$.
\end{thm}
\noindent Theorem \ref{gupta78bipartite_general} corresponds to Theorem \ref{konig16} in the case when $k=\Delta(G)$. Let $\mu(v)$ denote the maximum number of parallel edges incident to $v$. The following theorem is an extension of Vizing's theorem by Gupta \cite{gupta1974}:

\begin{thm}[Gupta \cite{gupta1974}]
\label{gupta74general}
Let $G=(V,E)$ be a multigraph. For $k\in \mathbf{Z}_{>0}$, there exists a color assignment $\pi:E\rightarrow [k]$ satisfying the following two conditions for every $v\in V$\textup{:}
\begin{itemize}
    \item $|\pi(\delta(v))|\geq \min\{\mathrm{deg}(v),k-\mu(v)\}$ holds if $\mathrm{deg}(v)\leq k$, and
    \item $|\pi(\delta(v))|\geq \min\{\mathrm{deg}(v)-\mu(v),k\}$ holds otherwise. 
\end{itemize}
\end{thm}
\noindent Theorem \ref{gupta74general} implies Theorem \ref{vizing65} in the case when $k=\Delta(G)+\mu(G)$. Theorem~\ref{gupta74general} was first announced by Gupta~\cite{gupta1974} without proof and subsequently proved by Fournier \cite{fournier1977}. Fournier's proof of Theorem~\ref{gupta74general} starts with any assignment of colors to the edges of $G$, and classifies the assignment into several cases, and finally shows the existence of a ``better'' assignment of colors in each case. 

Gupta \cite{gupta1974} provided another generalization of Theorem \ref{vizing65}, which was also subsequently proved by Fournier \cite{fournier1977}.
\begin{thm}[Gupta \cite{gupta1974}]
\label{gupta74general2}
Let $G=(V,E)$ be a multigraph. For $k\in \mathbf{Z}_{>0}$, suppose that $S=\{v\in V\mid \mathrm{deg}(v)+\mu(v) > k\}$ is a stable set. Then there exists a color assignment $\pi:E\rightarrow [k]$ such that $|\pi(\delta(v))|\geq \min\{\mathrm{deg}(v),k\}$ holds for every $v\in V$.
\end{thm}
\noindent Note that Theorem \ref{gupta74general2} coincides with Theorem \ref{vizing65} when $k=\Delta(G)+\mu(G)$. 

In this paper, we give the following generalization of Theorem \ref{gupta74general2}, which also implies Theorem \ref{gupta74general} in a certain sense.
\begin{thm}
\label{new-edge-coloringthm}
For a multigraph $G=(V,E)$ and $k\in \mathbf{Z}_{>0}$, let $c:V\rightarrow \mathbf{Z}_{+}$ be a function satisfying $c(v)\leq \min\{\mathrm{deg}(v),k\}$ for every $v\in V$. If $S=\{v\in V\mid c(v)+\mu(v) > k\}$ is a stable set, then there exists an assignment of colors $\pi:E\rightarrow [k]$ such that 
\begin{align}
\label{thm:edge_coloring_cd}
|\pi(\delta(v))|\geq c(v)
\end{align}
holds for every $v\in V$.
\end{thm}
\noindent In the case when $c(v)=\min\{\mathrm{deg}(v),k\}$ holds for every $v\in V$, Theorem \ref{new-edge-coloringthm} reduces to Theorem \ref{gupta74general2}. In addition, Theorem \ref{new-edge-coloringthm} yields an alternative proof of Theorem \ref{gupta74general} (see Section \ref{subsec:proof_of_gupta} for details).

\subsection{Supermodular extension of edge-coloring theorems}
Schrijver~\cite{schrijver1985} extended Theorems \ref{konig16} and \ref{gupta78bipartite_general} for bipartite multigraphs to a framework of supermodular functions on intersecting families. To describe this, we need some definitions. Let $U$ be a finite set. A pair of $X,Y\subseteq U$ is called an \textit{intersecting pair} (or $X$ and $Y$ are called \textit{intersecting}) if $X\cap Y\neq \emptyset$. A family $\mathcal{F}\subseteq 2^U$ is called an \textit{intersecting family} if $X\cup Y,X\cap Y\in \mathcal{F}$ holds for every intersecting pair $X,Y\in \mathcal{F}$. A function $g:\mathcal{F}\rightarrow \mathbf{R}$ is called \textit{intersecting supermodular} if $\mathcal{F}$ is an intersecting family and $g(X)+g(Y)\leq g(X\cup Y)+g(X\cap Y)$ holds for every intersecting pair $X,Y\in \mathcal{F}$. Schrijver \cite{schrijver1985} showed the following coloring-type theorem on an intersecting supermodular function.
\begin{thm}[Schrijver \cite{schrijver1985}]
\label{cor:supermo}
Let $\mathcal{F}\subseteq 2^U$ be an intersecting family and $g:\mathcal{F}\rightarrow \mathbf{Z}$ an intersecting supermodular function. For $k\in \mathbf{Z}_{>0}$, if $\min\{|X|,k\}\geq g(X)$ holds for each $X\in \mathcal{F}$, then there exists an assignment of colors $\pi:U\rightarrow [k]$ satisfying $|\pi(X)|\geq g(X)$ for each $X\in \mathcal{F}$.
\end{thm}
We are now ready to describe the supermodular coloring theorem, which is a generalization of Theorems \ref{konig16} and \ref{gupta78bipartite_general} to a framework of intersecting supermodular functions.

\begin{thm}[Schrijver \cite{schrijver1985}]
\label{thm:supermo}
Let $\mathcal{F}_1,\mathcal{F}_2\subseteq 2^U$ be intersecting families, and $g_1:\mathcal{F}_1\rightarrow \mathbf{Z}$ and $g_2:\mathcal{F}_2\rightarrow \mathbf{Z}$ intersecting supermodular functions. For $k\in \mathbf{Z}_{>0}$, if $\min\{|X|,k\}\geq g_i(X)$ holds for each $i=1,2$ and each $X\in \mathcal{F}_i$, then there exists an assignment of colors $\pi:U\rightarrow [k]$ such that $|\pi(X)|\geq g_i(X)$ holds for each $i=1,2$ and each $X\in \mathcal{F}_i$.
\end{thm}
\noindent Tardos \cite{tardos1985} gave an alternative proof of Theorem \ref{thm:supermo} using properties on generalized matroids. Theorem \ref{thm:supermo} was further extended to more general frameworks such as a framework including skew-supermodular coloring \cite{frank2014}, and a framework of list supermodular coloring \cite{iwata2018,yokoi2019}.

\begin{figure}[tb]
  \centering
  \includegraphics[width=14cm]{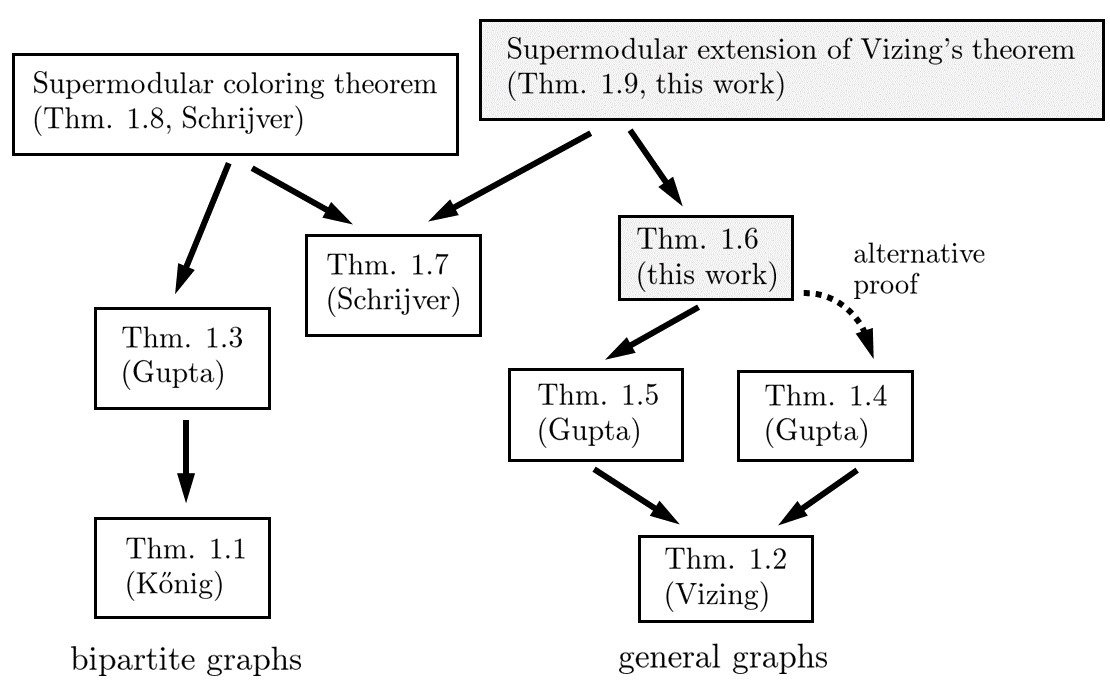}
  \caption{The relationship between the coloring-type theorems. The arrows mean implications.}
  \label{fig:relation}
\end{figure}

Figure \ref{fig:relation} describes the relationship between the above coloring-type theorems. A natural question arising from the supermodular coloring theorem is how to generalize Theorem \ref{vizing65} to a similar framework of supermodular functions.

Our main goal in this paper is to generalize Theorem \ref{vizing65} to a framework of a certain type of supermodular functions. In other words, we will provide a common generalization of Theorems \ref{vizing65} and \ref{cor:supermo}. To describe this, we need some definitions including new classes of intersecting families and intersecting supermodular functions. A family $\mathcal{F}\subseteq 2^U$ is called an \textit{intersecting 2/3-laminar family} if for every distinct $X_1,X_2,X_3\in \mathcal{F}$ satisfying $X_1\cap X_2\cap X_3\neq \emptyset$, there exist distinct two pairs $(i,j),(k,l)\in \{(1,2),(2,3),(3,1)\}$ such that $X_i\cup X_j,X_i\cap X_j,X_k\cup X_l,X_k\cap X_l\in \mathcal{F}$. A function $g:\mathcal{F}\rightarrow \mathbf{R}$ is called \textit{intersecting 2/3-supermodular} if $\mathcal{F}$ is an intersecting 2/3-laminar family and for every distinct 
$X_1,X_2,$ $X_3\in~ \mathcal{F}$ satisfying $X_1\cap X_2\cap X_3\neq \emptyset$, there exist distinct two pairs $(i,j),(k,l)\in \{(1,2),(2,3),(3,1)\}$ such that $X_i\cup X_j,X_i\cap X_j,X_k\cup X_l,X_k\cap X_l\in \mathcal{F}$ and
\begin{align}
    &g(X_i)+g(X_j)\leq g(X_i\cup X_j)+g(X_i\cap X_j),\notag\\
    &g(X_k)+g(X_l)\leq g(X_k\cup X_l)+g(X_k\cap X_l).\notag
\end{align}
The class of intersecting 2/3-supermodular functions is a common generalization of 2/3-supermodular functions and intersecting supermodular functions. A \textit{2/3-supermodular function} is a set function which satisfies the supermodular inequality for at least two pairs out of three pairs formed from every distinct three subsets. This class of functions was introduced in a separated paper \cite{mizutani}, and is a stronger version of 1/3-supermodular functions by B{\'{e}}rczi and Frank \cite{berczi2008}. There are some examples of intersecting 2/3-supermodular functions and their submodular variants, such as the rank function of a relaxation of sparse paving matroids, and a set function defined on a family $\{\delta(v)\mid v\in V\}$ for an undirected graph $G=(V,E)$. See Section \ref{sec:2/3-submo} for details on intersecting 2/3-supermodular functions. 

For a family $\mathcal{F}\subseteq 2^U$ and a function $g:\mathcal{F}\rightarrow \mathbf{R}$, a subfamily $\mathcal{L}\subseteq \mathcal{F}$ is called a \textit{$g$-laminar family} if for every pair of sets $X,Y\in \mathcal{L}$, at least one of the following two conditions holds.
\begin{itemize}
    \item At least one of $X\setminus Y,Y\setminus X,X\cap Y$ is the empty set.
    \item $X\cup Y,X\cap Y\in \mathcal{F}$ and $g(X)+g(Y)\leq g(X\cup Y)+g(X\cap Y)$ holds.
\end{itemize}
Since the first condition corresponds to the laminar family constraint, a $g$-laminar family is a relaxation of a laminar family. For $\mathcal{F}\subseteq 2^U$ and $X\in \mathcal{F}$, we define $D_{\mathcal{F}}(X)=\max\{|X\cap Y|\mid Y=\emptyset,\ \mathrm{or\ }Y\in \mathcal{F}\mathrm{\ and}\ X\not\subseteq Y\not\subseteq X\}$. We are now ready to describe a common generalization of Theorems \ref{vizing65} and \ref{cor:supermo}:
\begin{thm}
\label{thm:vizing_supermodular}
Let $\mathcal{F}\subseteq 2^U$ be an intersecting 2/3-laminar family and $g:\mathcal{F}\rightarrow \mathbf{Z}$ an intersecting 2/3-supermodular function. For $k\in \mathbf{Z}_{>0}$, suppose that $\mathcal{L}=\{X\in \mathcal{F}\mid g(X)+D_{\mathcal{F}}(X)>k\}$ is a $g$-laminar family and $\min\{|X|,k\}\geq g(X)$ holds for every $X\in \mathcal{F}$. Then there exists an assignment of colors $\pi:U\rightarrow [k]$ such that
\begin{align}
\label{thmeq:visupermo}
    |\pi(X)|\geq g(X)
\end{align}
holds for every $X\in \mathcal{F}$.
\end{thm}
\noindent Theorem \ref{thm:vizing_supermodular} also includes Theorem \ref{new-edge-coloringthm} as a special case. See Figure \ref{fig:relation} for the relationship between Theorem \ref{thm:vizing_supermodular} and other coloring theorems. The $g$-laminar family condition in Theorem \ref{thm:vizing_supermodular} generalizes the stable set condition in Theorem \ref{new-edge-coloringthm}.

The proof of Theorem \ref{thm:vizing_supermodular} constructs a desired coloring by repeating appropriate updates of the current coloring along with a ``bicolor chain'' and a proper sequence including an uncolored element. This construction comes from the proof technique of Theorem \ref{vizing65} called ``sequential recoloring'' by Berge and Fournier \cite{fournier1991}. The construction also uses an oracle for maximizing 2/3-supermodular functions. Due to the polynomial time algorithms to maximize 2/3-supermodular functions \cite{mizutani}, one can compute a desired coloring in polynomial time under some condition.
\begin{thm}
\label{thm:polytime}
A coloring in Theorem \ref{thm:vizing_supermodular} can be obtained in polynomial time if $\mathcal{F}=2^U$.
\end{thm}

\subsection{Organization of the paper}
The rest of this paper is organized as follows. Section \ref{sec:2/3-submo} is an introduction for intersecting 2/3-supermodular functions, and describes the relationship between intersecting 2/3-supermodular functions and other function classes. Section \ref{chap:edge-coloring} provides a proof of Theorem \ref{new-edge-coloringthm}, which is based on the proof of Theorem \ref{vizing65} by Berge and Fournier \cite{fournier1991}. Section \ref{chap:edge-coloring} also gives an alternative proof of Theorem \ref{gupta74general} using Theorem \ref{new-edge-coloringthm}. Section \ref{chap:supermodular_extension} provides a proof of Theorem \ref{thm:vizing_supermodular} combining the proof technique of Theorem \ref{vizing65} by Berge and Fournier~\cite{fournier1991}, and that of Theorems \ref{cor:supermo} and \ref{thm:supermo} by Schirijver \cite{schrijver1985}. Section \ref{sec:impl} proves that Theorem \ref{thm:vizing_supermodular} includes Theorems \ref{new-edge-coloringthm} and \ref{cor:supermo} as special cases. Section \ref{sec:polyalgo} shows that the construction in the proof of Theorem \ref{thm:vizing_supermodular} yields a polynomial time algorithm to obtain a desired coloring under a certain condition with the aid of polynomial algorithms to maximize 2/3-supermodular functions \cite{mizutani}.

\section{Intersecting 2/3-supermodular functions}
\label{sec:2/3-submo}
Let $U$ be a finite set. A set function $f:2^U\rightarrow \mathbf{R}$ is called \textit{submodular} if the \textit{submodular inequality} $f(X)+f(Y)\geq f(X\cup Y)+f(X\cap Y)$ holds for any pair of sets $X,Y\subseteq U$. B{\'{e}}rczi and Frank \cite{berczi2008} introduced \textit{1/3-submodular functions} $f:2^U\rightarrow \mathbf{R}$, which satisfy the submodular inequality for at least one pair out of three pairs formed from every distinct three subsets. The class of 1/3-submodular functions includes the minimum of two matroid rank functions \cite{barasz2006,berczi2023}, and the minimum of two submodular functions. If $f$ satisfies the submodular inequality for at least two pairs out of three pairs formed from every distinct three subsets, then $f$ is called \textit{2/3-submodular} \cite{mizutani}. The class of 1/3-submodular functions includes the class of 2/3-submodular functions, and the class of 2/3-submodular functions includes the class of submodular functions. One example of 2/3-submodular functions is a relaxation of rank functions of sparse paving matroids.
\begin{ex}
Consider the base family $\mathcal{B}$ of a uniform matroid on the ground set $U$ with rank $k$ such that $2\leq k\leq |U|-2$. Let $\mathcal{F}\subseteq \mathcal{B}$ be a family such that any two distinct sets $X,Y\in \mathcal{F}$ satisfy $|X\cap Y|\leq k-2$. Then, a matroid with a base family $\mathcal{B}\setminus \mathcal{F}$ is called a \textit{sparse paving matroid}. Let $\mathcal{F}'\subseteq \mathcal{B}$ be a family such that any three distinct sets $X,Y,Z\in \mathcal{F}'$ satisfy $\min\{|X\cap Y|,|X\cap Z|\}\leq k-2$. Define a rank function $r:2^U\rightarrow \mathbf{Z}$ of $\mathcal{B}\setminus \mathcal{F}'$ in the same manner as matroid rank functions:
\begin{align}
r(X)=\max\{|X\cap B|\mid B\in \mathcal{B}\setminus \mathcal{F}'\}\ \ \ (X\subseteq U).\notag
\end{align}
Then, $r$ is not necessarily submodular but 2/3-submodular. One can show that $r$ is 2/3-submodular as follows. For any set $X\subseteq U$ with $|X|=k-1$, it holds that $r(X)=|X|$ because for distinct three elements $u_1,u_2,u_3\in U\setminus X$, at least one of $X\cup \{u_1\},X\cup \{u_2\},X\cup \{u_3\}$ is included in $\mathcal{B}\setminus \mathcal{F}'$ by the definition of $\mathcal{F}'$. This implies $r(X)=|X|$ for any $X\subseteq U$ with $|X|\leq k-1$. Similarly, for any $X\subseteq U$ with $|X|=k+1$, it holds that $r(X)=k$ because for distinct three elements $u_1,u_2,u_3\in X$, at least one of $X\setminus \{u_1\},X\setminus \{u_2\},X\setminus \{u_3\}$ is included in $\mathcal{B}\setminus \mathcal{F}'$ by the definition of $\mathcal{F}'$. This also implies that $r(X)=k$ for any $X\subseteq U$ with $|X|\geq k+1$. For any $X\subseteq U$ with $|X|=k$, it holds that $r(X)=k-1$ if $X\in \mathcal{F}'$, and $r(X)=k$ if $X\notin \mathcal{F}'$. Hence, if $X,Y\notin \mathcal{F}'$, then $r(X)+r(Y)\geq r(X\cup Y)+r(X\cap Y)$ holds by the submodularity of matroid rank functions. Consider the case when $X\in \mathcal{F}'$. Since $r(X)=k-1$, if $r(X)+r(Y)< r(X\cup Y)+r(X\cap Y)$, then $r(X\cup Y)=k$ and $r(Y)=r(X\cap Y)$ hold, which implies that $Y\setminus X\ne \emptyset$ and $k-1=r(X)\geq r(X\cap Y)=r(Y)$. Then we have $|Y|>|X\cap Y|$ and $k-1\geq r(Y)=r(X\cap Y)$, which implies that $Y\in \mathcal{F}'$ and $|X\cap Y|=k-1$. Therefore, the submodular inequality of $r$ does not hold only for $X,Y\in \mathcal{F}'$ with $|X\cap Y|=k-1$. Thus, $r$ is 2/3-submodular by the definition of $\mathcal{F}'$.
\end{ex}
A set function $g$ is called \textit{supermodular} (resp.\textit{ 1/3-supermodular}, \textit{ 2/3-supermodular}) if $-g$ is submodular (resp. 1/3-submodular, 2/3-submodular). There are some intersecting variants of these supermodular functions. A family $\mathcal{F}\subseteq 2^U$ is called an \textit{intersecting family} if every pair $X,Y\in \mathcal{F}$ with $X\cap Y\ne \emptyset$ satisfies $X\cup Y,X\cap Y\in \mathcal{F}$. A set function $g:\mathcal{F}\rightarrow \mathbf{R}$ is called \textit{intersecting supermodular} if $\mathcal{F}$ is an intersecting family and $g(X)+g(Y)\leq g(X\cup Y)+g(X\cap Y)$ holds for every pair $X,Y\in \mathcal{F}$ with $X\cap Y\ne \emptyset$. Schrijver \cite{schrijver1985} generalized K\H{o}nig's edge-coloring theorem \cite{konig1916} to a framework defined with two intersecting supermodular functions. To describe a supermodular extension of Vizing's edge-coloring theorem \cite{vizing1965}, this paper introduces an intersecting version of 2/3-supermodular functions. A family $\mathcal{F}\subseteq 2^U$ is called an \textit{intersecting 2/3-laminar family} if for every distinct three sets $X_1,X_2,X_3\in \mathcal{F}$ with $X_1\cap X_2\cap X_3\ne \emptyset$, there exist distinct two pairs of indices $\{i,j\},\{k,l\}\subseteq \{1,2,3\}$ such that $X_i\cup X_j,X_i\cap X_j,X_k\cup X_l,X_k\cap X_l\in \mathcal{F}$. A set function $g:\mathcal{F}\rightarrow \mathbf{R}$ is called \textit{intersecting 2/3-supermodular} if $\mathcal{F}$ is an intersecting 2/3-laminar family and for every distinct three sets $X_1,X_2,X_3\in \mathcal{F}$ with $X_1\cap X_2\cap X_3\ne \emptyset$, there exist distinct two pairs of indices $\{i,j\},\{k,l\}\subseteq \{1,2,3\}$ such that $X_i\cup X_j,X_i\cap X_j,X_k\cup X_l,X_k\cap X_l\in \mathcal{F}$ and $f(X_i)+f(X_j)\leq f(X_i\cup X_j)+f(X_i\cap X_j),\ f(X_k)+f(X_l)\leq f(X_k\cup X_l)+f(X_k\cap X_l)$. The class of intersecting 2/3-supermodular functions includes 2/3-supermodular functions and intersecting supermodular functions. The following is an example of intersecting 2/3-supermodular functions.
\begin{ex}
Let $G=(V,E)$ be a multigraph. Define $\mathcal{F}=\{\delta(v)\mid v\in V\}$, where $\delta(v)$ denotes the set of edges incident with $v$. Then a set function $g:\mathcal{F}\rightarrow \mathbf{R}$ is intersecting 2/3-supermodular regardless of the values of $g$ because any three distinct sets $X,Y,Z\in \mathcal{F}$ satisfy $X\cap Y\cap Z=\emptyset$.
\end{ex}
In the value oracle model, while it requires an exponential number of oracle calls to minimize 1/3-submodular functions \cite{berczi2008}, 2/3-submodular functions can be minimized in polynomial time using the ellipsoid method \cite{mizutani}.
\begin{thm}[\cite{mizutani}]
\label{thm:min23}
Let $f:2^U\rightarrow \mathbf{Z}$ be an integer-valued 2/3-submodular function. Then a minimizer of $f$ can be computed in polynomial time in $|U|$ and $\log B$, where $B$ is an upper bound of the absolute values of $f$.
\end{thm}
Theorem \ref{thm:min23} yields a polynomial time algorithm to obtain a coloring of Theorem \ref{thm:vizing_supermodular}. See Section \ref{sec:polyalgo} for details.

\section{An extension of Gupta's theorem}
\label{chap:edge-coloring}

\subsection{Proof of Theorem \ref{new-edge-coloringthm}}
\label{subsec:new_edge_color_proof}
Though we will prove a generalization of Theorem \ref{new-edge-coloringthm} in Section \ref{chap:supermodular_extension}, here we present the proof of Theorem 1.6 because it can be described with only graph terminology and may be of independent use. Similar to the proof of Theorem \ref{vizing65} by Berge and Fournier \cite{fournier1991}, the modification of a coloring called ``sequential recoloring'' plays an important role in the proof of Theorem \ref{new-edge-coloringthm}. 
\begin{proof}[Proof of Theorem \ref{new-edge-coloringthm}]
Let $F\subseteq E$ be a maximum subset such that there exists a color assignment $\pi:F\rightarrow [k]$
satisfying 
\begin{align}
\label{edge_coloring_cd}
    |\delta(v)\setminus F|+|\pi(\delta(v)\cap F)|\geq c(v)
\end{align}
for every $v\in V$. Such a set $F$ does exist because $F=\emptyset$ satisfies (\ref{edge_coloring_cd}). Our aim is to show that $F=E$, which implies that (\ref{edge_coloring_cd}) coincides with (\ref{thm:edge_coloring_cd}). Suppose for a contradiction that $F\neq E$. Take $e_0\in E\setminus F$ and let $x$ and $y_0$ be the endpoints of $e_0$. Since $S$ is a stable set, without loss of generality we may assume that $y_0\notin S$. If $y_0$ satisfies (\ref{edge_coloring_cd}) with strict inequality, then extend the domain $F$ of $\pi$ to $F'=F\cup \{e_0\}$ and set $\pi(e_0)=\alpha$ for $\alpha\notin \pi(\delta(x)\cap F)$ (if $\left|\pi(\delta(x)\cap F)\right|=k$, then set $\pi(e_0)$ as any color). This extended color assignment $\pi$ satisfies $|\delta(v)\setminus F'|+|\pi(\delta(v)\cap F')|\geq c(v)$ for every $v\in V$, which contradicts the maximality of $F$. So we may assume that $y_0$ satisfies (\ref{edge_coloring_cd}) with equality. Let $\{e_0,e_1,\ldots,e_l\}$ be a maximal sequence of distinct edges incident to $x$ satisfying the following five conditions, where $\alpha_i=\pi(e_i)$ for each $i\in [l]$, and $y_i$ is the endpoint of $e_i$ other than $x$ for each $i\in [l]$:
\begin{enumerate}
    \item $e_i\in F$ for every $i\in [l]$.
    \item $\alpha_i\notin \pi(\delta(y_{i-1})\cap F)$ for every $i\in [l]$.
    \item $\alpha_{i+1}\neq \alpha_{j+1}$ for every $0\leq i\neq j\leq l-1$ with $y_i=y_j$.
    \item $y_0,\ldots,y_l\notin S$.
    \item (\ref{edge_coloring_cd}) is satisfied with equality when $v=y_0,\ldots,y_l$.
\end{enumerate}
Such a sequence does exist because the sequence $\{e_0\}$ satisfies the above conditions. Since $y_l\notin S$, we have $c(y_l)+\mu(y_l)\leq k$. Since $y_l$ satisfies (\ref{edge_coloring_cd}) with equality, we have
\begin{align}
    |\pi(\delta(y_l)\cap F)|+\mu(y_l)\leq |\delta(y_l)\setminus F|+|\pi(\delta(y_l)\cap F)|+\mu(y_l)=c(y_l)+\mu(y_l)\leq k.\notag
\end{align}
This implies that $k-|\pi(\delta(y_l)\cap F)|\geq \mu(y_l)$. That is, the number of colors not
contained in $\pi(\delta(y_l)\cap F)$ is at least $\mu(y_l)$. So there exists a color $\alpha_{l+1}\notin \pi(\delta(y_l)\cap F)$ satisfying $\alpha_{l+1}\neq \alpha_{i+1}$ for every $i<l$ with $y_i=y_l$. If $\pi(\delta(x)\cap F)$ contains all of the $k$ colors, then extend the domain $F$ of $\pi$ to $F'$ defined above and set $\pi(e_0)=\alpha_1\notin \pi(\delta(y_0)\cap F)$. This extended color assignment $\pi$ satisfies $|\delta(v)\setminus F'|+|\pi(\delta(v)\cap F')|\geq c(v)$ for every $v\in V$, which contradicts the maximality of $F$. Hence, we may assume that there exists a color $\beta\notin \pi(\delta(x)\cap F)$. Consider the case when $\beta\notin \pi(\delta(y_l)\cap F)$. Define the following color assignment $\pi_l:F'\rightarrow [k]$:
\begin{align}
    \pi_l(e)=\left\{
\begin{array}{lll}
\alpha_{i+1} & (e=e_i,\ 0\leq i\leq l-1),\\
\beta & (e=e_l),\\
\pi(e) & (\mathrm{otherwise}).
\end{array}
\right.\notag
\end{align}
Then we have $|\delta(v)\setminus F'|+|\pi_l(\delta(v)\cap F')|\geq c(v)$ for every $v\in V$, which contradicts the maximality of $F$. So we may assume that $\beta\in \pi(\delta(y_l)\cap F)$. Let $P=(f_1,\ldots,f_p)$ be a maximal trail starting at $y_l$ such that $f_i\in F$ and $\pi(f_i)\in \{\alpha_{l+1},\beta\}$ hold for each $i\in [p]$, and $\pi(f_i)\neq \pi(f_{i+1})$ holds for each $i\in [p-1]$ (since $\beta\in \pi(\delta(y_l)\cap F)$, $P$ consists of at least one edge). Then $P$ satisfies one of the following two conditions (see Figure \ref{fig:new_edgecolor}):
\begin{enumerate}[label=(\alph*),ref=\alph*]
    \item \label{P_cd1}The endpoint $t$ of $P$ other than $y_l$ satisfies either $\alpha_{l+1}\notin \pi(\delta(t)\cap F)$ or $\beta\notin \pi(\delta(t)\cap F)$.
    \item \label{P_cd2}There exist two edges $e_1,e_2\in \delta(t)\cap F\cap P$ satisfying $\{\pi(e_1),\pi(e_2)\}=\{\alpha_{l+1},\beta\}$.
\end{enumerate}

\begin{figure}[tb]
  \centering
  \includegraphics[width=12cm]{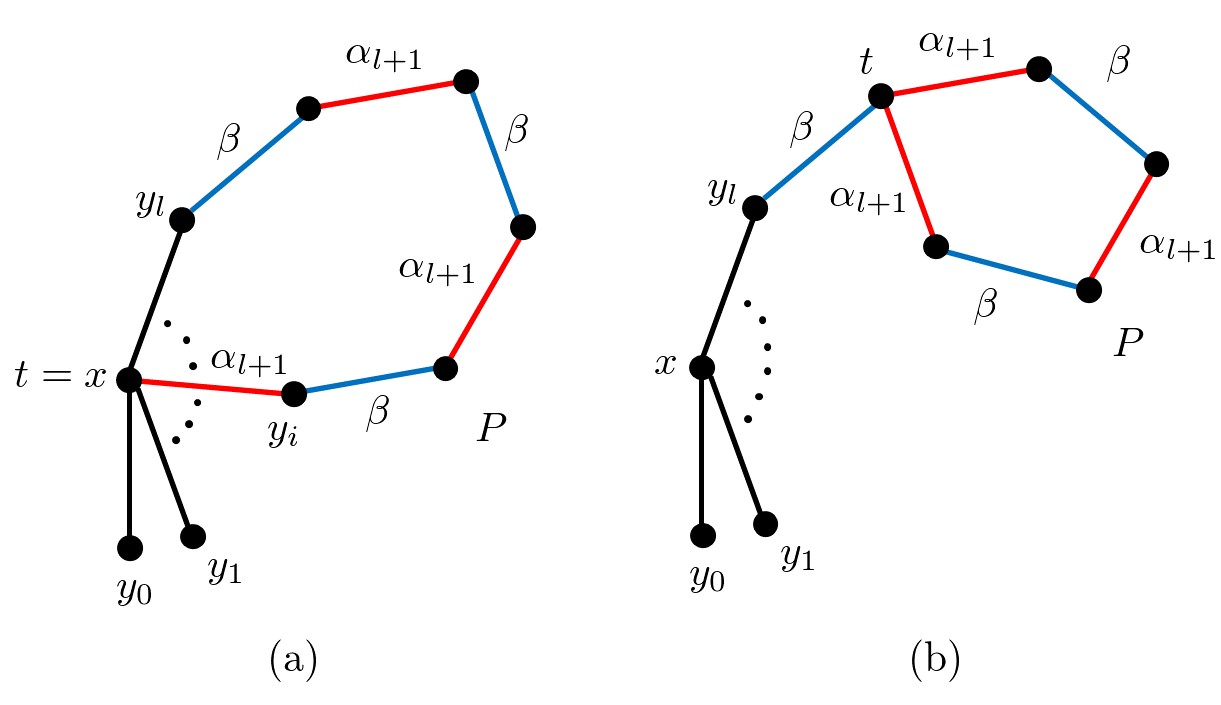}
  \caption{A maximal trail $P$ satisfying the condition (\ref{P_cd1}), and that satisfying the condition (\ref{P_cd2}).}
  \label{fig:new_edgecolor}
\end{figure}

We now consider three cases (I) $t\neq x,y_0,\ldots,y_{l-1}$, (II) $t=y_i$ for some $0\leq i\leq l-1$, and (III) $t=x$ separately.
\subsubsection*{\underline{(I) $\bm{t\neq x,y_0,\ldots,y_{l-1}.}$}}
Define the following color assignment $\pi':F'\rightarrow ~[k]$:
\begin{align}
    \pi'(e)=\left\{
\begin{array}{lllll}
\beta & (e\in P\mathrm{\ and}\ \pi(e)=\alpha_{l+1},\ \mathrm{or\ }e=e_l),\\
\alpha_{l+1} & (e\in P,\ \pi(e)=\beta),\\
\alpha_{i+1} & (e=e_i\ \mathrm{for\ some\ }0\leq i\leq l-1),\\
\pi(e) & (\mathrm{otherwise}).
\end{array}
\right.\notag
\end{align}
Then we have $|\pi'(\delta(v)\cap F')|\geq |\pi(\delta(v)\cap F)|$ for every $v\in V$ because $t\neq x,y_0,\ldots,y_{l-1}$ and $t$ satisfies one of (\ref{P_cd1}) and (\ref{P_cd2}). We also have $|\pi'(\delta(x)\cap F')|> |\pi(\delta(x)\cap F)|$ and $|\pi'(\delta(y_0)\cap F')|> |\pi(\delta(y_0)\cap F)|$. Hence, we have $|\delta(v)\setminus F'|+|\pi'(\delta(v)\cap F')|\geq c(v)$ for every $v\in V$, which contradicts the maximality of $F$. 

\subsubsection*{\underline{(II) $\bm{t=y_i}$ \textbf{for some} $\bm{0\leq i\leq l-1}.$}}
We consider two cases $\alpha_{l+1}\in \pi(\delta(y_i)\cap F)$ and $\alpha_{l+1}\notin \pi(\delta(y_i)\cap F)$ separately.
\begin{itemize}
    \item Consider the case when $\alpha_{l+1}\in \pi(\delta(y_i)\cap F)$. Then we have $\alpha_{j+1}\neq \alpha_{l+1},\beta$ for every $j$ with $y_j=y_i$, which implies that $|\pi'(\delta(y_i)\cap F')|\geq  |\pi(\delta(y_i)\cap F)|$ holds for $\pi'$ and $F'$ defined above because one of the conditions (\ref{P_cd1}) and (\ref{P_cd2}) holds. Hence, we have $|\delta(v)\setminus F'|+|\pi'(\delta(v)\cap F')|\geq c(v)$ for every $v\in V$, which contradicts the maximality of $F$.
    \item Consider the case when $\alpha_{l+1}\notin \pi(\delta(y_i)\cap F)$. Then we have $\pi(f_p)=\beta$. Define the following color assignment $\pi_i:F'\rightarrow ~[k]$:
\begin{align}
    \pi_i(e)=\left\{
\begin{array}{lllll}
\beta & (e\in P\mathrm{\ and}\ \pi(e)=\alpha_{l+1},\ \mathrm{or\ }e=e_i),\\
\alpha_{l+1} & (e\in P,\ \pi(e)=\beta),\\
\alpha_{j+1} & (e=e_j\ \mathrm{for\ some\ }0\leq j\leq i-1),\\
\pi(e) & (\mathrm{otherwise}).\notag
\end{array}
\right.
\end{align}
Then we have $|\pi_i(\delta(y_i)\cap F')|\geq  |\pi(\delta(y_i)\cap F)|$ because $\alpha_{l+1}\notin \pi(\delta(y_i)\cap F)$ and $\pi(f_p)=\beta$. Hence, we have $|\delta(v)\setminus F'|+|\pi_i(\delta(v)\cap F')|\geq c(v)$ for every $v\in V$, which contradicts the maximality of $F$.
\end{itemize}

\subsubsection*{\underline{(III) $\bm{t=x}.$}}
Let $y$ be the endpoint of $f_p$ other than $x$. We consider two cases $y\neq y_0,\ldots,y_{l-1}$ and $y=y_i$ for some $0\leq i\leq l-1$ separately.
\begin{itemize}
    \item Consider the case when $y\neq y_0,\ldots,y_{l-1}$. Since $\beta\notin \pi(\delta(x)\cap F)$, we have $\pi(f_p)=\alpha_{l+1}$. Then by the maximality of the sequence $\{e_0,\ldots,e_l\}$, $y$ satisfies $y\in S$ or (\ref{edge_coloring_cd}) with strict inequality. Consider the case when $y\in S$. Define the following color assignment $\pi_x:F'\setminus \{f_p\}\rightarrow [k]$:
    \begin{align}
    \pi_x(e)=\left\{
    \begin{array}{ll}
    \alpha_{i+1} & (e=e_i\ \mathrm{for\ some\ }0\leq i\leq l),\\
    \pi(e) & (\mathrm{otherwise}).
    \end{array}
    \right.\notag
    \end{align}
    Then we have $|\delta(v)\setminus (F'\setminus \{f_p\})|+|\pi_x(\delta(v)\cap (F'\setminus \{f_p\}))|\geq c(v)$ for every $v\in V$. We now redefine $x,y_0,e_0,\pi,F$ to be $y,x,f_p,\pi_x,F'\setminus \{f_p\}$, respectively. Then we again start from the beginning of this proof with the redefined $x,y_0,e_0,\pi,F$. Since the redefined $x$ satisfies $x\in S$, every neighbor $v$ of $x$ satisfies $v\notin S$ in the redefined setting. So we may assume that $y\notin S$. This implies that $y$ satisfies (\ref{edge_coloring_cd}) with strict inequality. Define the following color assignment $\pi_x':F'\rightarrow [k]$:
     \begin{align}
    \pi_x'(e)=\left\{
    \begin{array}{lll}
    \alpha_{i+1} & (e=e_i\ \mathrm{for\ some\ }0\leq i\leq l),\\
    \beta & (e=f_p),\\
    \pi(e) & (\mathrm{otherwise}).
    \end{array}
    \right.\notag
    \end{align}
    Then we have $|\delta(v)\setminus F'|+|\pi_x'(\delta(v)\cap F')|\geq c(v)$ for every $v\in V$, which contradicts the maximality of $F$.
    \item Consider the case when $y=y_i$ for some $0\leq i\leq l-1$. Similar to the first case, we have $\pi(f_p)=\alpha_{l+1}$. Then by the maximality of the sequence $\{e_0,\ldots,e_l\}$, we have $f_p=e_i$ for some $1\leq i\leq l-1$. Define the following color assignment $\pi_x'':F'\rightarrow [k]$:
    \begin{align}
    \pi_x''(e)=\left\{
    \begin{array}{llll}
    \beta & (e\in P,\ \pi(e)=\alpha_{l+1}),\\
    \alpha_{l+1} & (e\in P,\ \pi(e)=\beta),\\
    \alpha_{j+1} & (e=e_j\ \mathrm{for\ some\ }0\leq j\leq i-1),\\
    \pi(e) & (\mathrm{otherwise}).
    \end{array}
    \right.\notag
    \end{align}
    Then we have $|\delta(v)\setminus F'|+|\pi_x''(\delta(v)\cap F')|\geq c(v)$ for every $v\in V$, which contradicts the maximality of $F$.
\end{itemize}
\end{proof}

\subsection{Proving Theorem \ref{gupta74general} via Theorem \ref{new-edge-coloringthm}}
\label{subsec:proof_of_gupta}
In this section, we give an alternative proof of Theorem \ref{gupta74general} using Theorem \ref{new-edge-coloringthm}.

\begin{proof}[Proof of Theorem \ref{gupta74general}]
Our aim is to construct a color assignment $\pi$ satisfying the two conditions in Theorem \ref{gupta74general}. To construct such an assignment, we first appropriately orient some of edges in $G$. Then we assign colors to undirected edges in $G$ using Theorem \ref{new-edge-coloringthm}. After that, for each directed edge $e$ entering a vertex $v$, we successively assign a color $\alpha$ to $e$ such that $\alpha$ does not occur on (already colored) undirected and directed edges incident to $v$. By this procedure, we obtain the desired color assignment satisfying the two conditions in Theorem \ref{gupta74general}.

Let $W=\{v\in V\mid \mathrm{deg}_G(v)\geq k+1\}$. For $U\subseteq V$, the \textit{induced subgraph} $G[U]$ is the graph with the vertex set $U$ and the edge set consisting of all edges spanned by $U$. To orient some of the edges of $G$, we execute the following algorithm:
\vspace{10pt}
\\
\textbf{Edge-orientation algorithm}
\vspace{-2pt}
\begin{description}
    \item[\textbf{Step 1.}]\label{eoalgo:step1} Take a cycle consisting of undirected edges in $G[W]$, or a path $P$ consisting of undirected edges in $G[W]$ such that each endpoint of $P$ is not incident to undirected edges in $G[W]$ except for edges in $P$ (see Figure \ref{fig:edge_orientation}). Orient edges of this cycle or path in the same direction.
    \item[\textbf{Step 2.}]\label{eoalgo:step2} For each $v\in W$ satisfying $\delta_G^+(v)=\min\{\mathrm{deg}_G(v)-k,\ \mu_G(v)\}$, update $W:=W\setminus \{v\}$.
    \item[\textbf{Step 3.}]\label{eoalgo:step3} If there exist no undirected edges in $G[W]$, then terminate the algorithm. Otherwise, go back to Step 1.
\end{description}
\begin{figure}[tb]
  \centering
  \includegraphics[width=10cm]{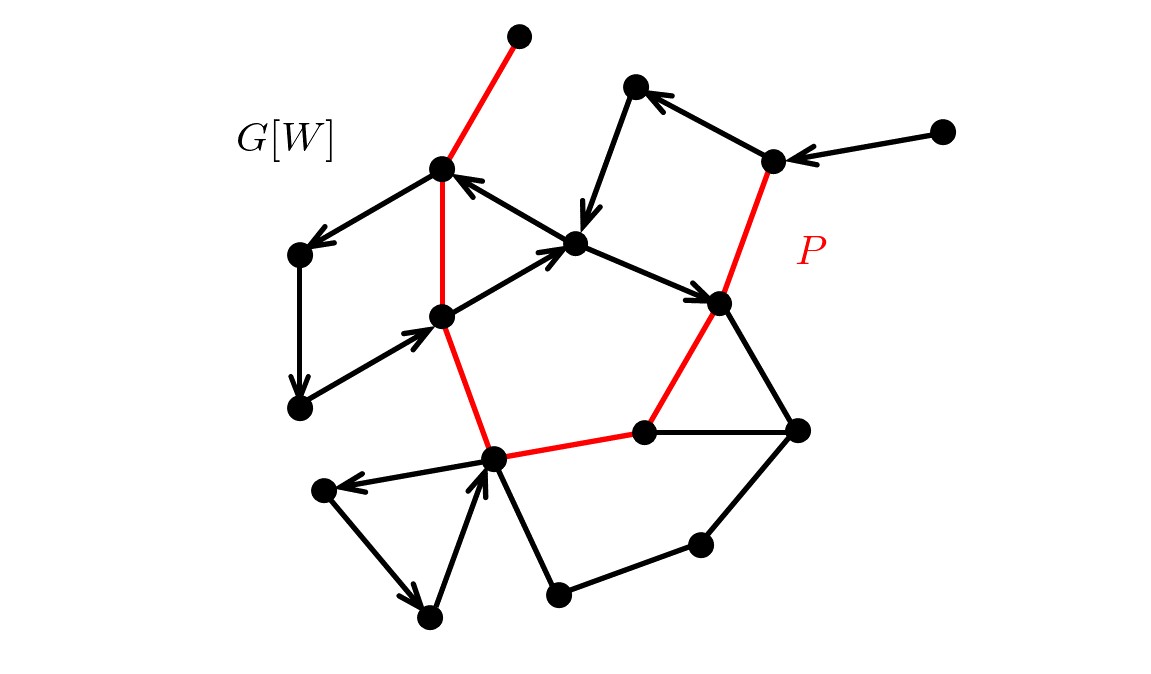}
  \caption{A path $P$ in Step 1. of the edge-orientation algorithm.}
  \label{fig:edge_orientation}
\end{figure}

Note that $\delta_G^+(v)$ denotes the number of directed edges leaving $v$, and $\mathrm{deg}_G(v)$ denotes the number of edges (undirected or directed edges) incident to $v$. In addition, $\mu_G(v)$ denotes the maximum number of edges (undirected or directed edges) between $v$ and any vertex in $G$. We now execute the algorithm, and let $H$ be a multigraph obtained by deleting all directed edges in $G$. Define
\begin{equation}
    c(v)=\begin{cases}
        \min\{\mathrm{deg}_G(v),\ k-\mu_G(v)\}& (\mathrm{deg}_G(v)\leq k),\\
        \min\{\mathrm{deg}_G(v)-\mu_G(v),\ k\}-\delta_G^{-}(v)& (\mathrm{deg}_G(v)\geq k+1),
    \end{cases}\notag
\end{equation}
for every $v\in V$, where $\delta_G^-(v)$ denotes the number of directed edges entering $v$. Suppose that there exists a color assignment $\pi:E(H)\rightarrow [k]$ satisfying $|\pi(\delta_H(v))|\geq c(v)$ for every $v\in V$, where $E(H)$ denotes the set of all edges in $H$. Then, assign a color $\pi(e)$ to $e$ for each edge $e\in E(H)$. After that, assign a color $\alpha$ to $e$ for each directed edge $e\in E\setminus E(H)$ entering $v\in V$ one by one so that $\alpha$ does not occur on already colored undirected or directed edges incident to $v$ (if such a color $\alpha$ does not exist, then assign any color to $e$). Since  we have $|\pi(\delta_H(v))|\geq c(v)=\min\{\mathrm{deg}_G(v),\ k-\mu_G(v)\}$ for every $v\in V$ with $\mathrm{deg}_G(v)\leq k$, and we have $|\pi(\delta_H(v))|+\delta_G^{-}(v)\geq c(v)+\delta_G^{-}(v)=\min\{\mathrm{deg}_G(v)-\mu_G(v),\ k\}$ for every $v\in V$ with $\mathrm{deg}_G(v)\geq k+1$, the above color assignment satisfies the two conditions in Theorem \ref{gupta74general}. Hence, it suffices to show that we can apply Theorem \ref{new-edge-coloringthm} to $H$ and $c$ defined above. In other words, it suffices to prove that $c(v)\leq \min\{\mathrm{deg}_H(v),\ k\}$ for every $v\in V$, and $S=\{v\in V\mid c(v)+\mu_H(v)>k\}$ is a stable set in $H$. We first show the former inequality. If $\mathrm{deg}_G(v)\leq k$, then we have $v\notin W$, which implies $\mathrm{deg}_H(v)=\mathrm{deg}_G(v)$, and hence we have
\begin{align}
    \min\{\mathrm{deg}_H(v),\ k\}=\min\{\mathrm{deg}_G(v),\ k\}\geq \min\{\mathrm{deg}_G(v),\ k-\mu_G(v)\}=c(v).\notag
\end{align}
Consider the case when $\mathrm{deg}_G(v)\geq k+1$. Then we have $v\in W$, which implies that $\delta_G^+(v)\leq \min\{\mathrm{deg}_G(v)-k,\ \mu_G(v)\}$ by Step 2. of the algorithm. Hence, we have
\begin{align}
    \mathrm{deg}_H(v)&=\mathrm{deg}_G(v)-\delta_G^{-}(v)-\delta_G^{+}(v)\geq \mathrm{deg}_G(v)-\delta_G^{-}(v)-\min\{\mathrm{deg}_G(v)-k,\ \mu_G(v)\}\notag\\ 
    &=\max\{k,\ \mathrm{deg}_G(v)-\mu_G(v)\}-\delta_G^{-}(v)\geq \min\{k,\ \mathrm{deg}_G(v)-\mu_G(v)\}-\delta_G^{-}(v)\notag\\
    &=c(v).\notag
\end{align}
Also, we have $k\geq \min\{\mathrm{deg}_G(v)-\mu_G(v),\ k\}-\delta_G^{-}(v)=c(v)$. Therefore, it holds that $\min\{\mathrm{deg}_H(v),\ k\}\geq c(v)$.

We next show that $S$ is a stable set in $H$. Suppose for a contradiction that there exists an edge $e_{uv}\in E(H)$ connecting $u,v\in S$. If $\mathrm{deg}_G(u)\leq k$, then we have
\begin{align}
    c(u)+\mu_H(u)=\min\{\mathrm{deg}_G(u),\ k-\mu_G(u)\}+\mu_H(u)\leq k-\mu_G(u)+\mu_H(u)\leq k,\notag
\end{align}
which contradicts $u\in S$. So we may assume that $u,v\in W$. Since $e_{uv}\in E(H)$, $e_{uv}$ is not oriented when the algorithm terminates. This implies that at least one of $u$ and $v$ is deleted from $W$ in the algorithm. Without loss of generality, we may assume that $v$ is deleted from $W$ (and deleted not later than $u$ if $u$ is also deleted from $W$ in the algorithm). Then we have $\delta_G^+(v)=\min\{\mathrm{deg}_G(v)-k,\ \mu_G(v)\}$. In the algorithm, $G[W]$ contains the undirected edge $e_{uv}$ while $v\in W$. This implies that $v$ cannot be an endpoint of a path $P$ taken in Step 1. of the algorithm because $e_{uv}\notin P$. Hence, when the algorithm terminates, we have $\delta_G^-(v)=\delta_G^+(v)=\min\{\mathrm{deg}_G(v)-k,\ \mu_G(v)\}$. Therefore, we have
\begin{align}
    c(v)+\mu_H(v)&=\min\{\mathrm{deg}_G(v)-\mu_G(v),\ k\}-\delta_G^-(v)+\mu_H(v)\notag\\
    &=\min\{\mathrm{deg}_G(v)-\mu_G(v),\ k\}-\min\{\mathrm{deg}_G(v)-k,\ \mu_G(v)\}+\mu_H(v)\notag\\
    &=\min\{\mathrm{deg}_G(v)-\mu_G(v),\ k\}+\max\{k-\mathrm{deg}_G(v),\ -\mu_G(v)\}\notag\\
    &+\mathrm{deg}_G(v)-\mathrm{deg}_G(v)+\mu_H(v)\notag\\
    &=\min\{\mathrm{deg}_G(v)-\mu_G(v),\ k\}+\max\{k,\ \mathrm{deg}_G(v)-\mu_G(v)\}\notag\\
    &-\mathrm{deg}_G(v)+\mu_H(v)\notag\\
    &=k+\mathrm{deg}_G(v)-\mu_G(v)-\mathrm{deg}_G(v)+\mu_H(v)\notag\\
    &\leq k,\notag
\end{align}
which contradicts $v\in S$.
\end{proof}

\section{Proof of the supermodular extension}
\label{chap:supermodular_extension}
In this section, we give a proof of Theorem \ref{thm:vizing_supermodular}. The proof combines the proof technique of Theorems \ref{cor:supermo} and \ref{thm:supermo} by Schrijver \cite{schrijver1985}, and that of Theorem \ref{vizing65} by Berge and Fournier \cite{fournier1991}, which is called ``sequential recoloring''. The outline of the proof of Theorem \ref{thm:vizing_supermodular} is based on that of Theorem \ref{vizing65} by Berge and Fournier \cite{fournier1991} and Theorem \ref{new-edge-coloringthm}. The proof starts with taking a proper coloring of a maximum subset, and finds a sequence of elements starting from an uncolored element which will be sequentially recolored, and also finds a bicolor sequence of elements starting from a neighborhood (in some sense) of the last element of the previous sequence, and finally recolors elements appropriately along with those sequences including the uncolored element, which will contradict the maximality of the subset. We first prepare some useful lemmas for the proof. Let $S\subseteq U$ be a subset and $\pi:S\rightarrow [k]$ a color assignment. Define $f_\pi:\mathcal{F}\rightarrow \mathbf{Z}$ as $f_{\pi}(X)=|X\setminus S|+|\pi(X\cap S)|$ for each $X\in \mathcal{F}$. The following claim shows the submodularity of $f_{\pi}$.
\begin{clm}
\label{clm:submo}
If $X_1,X_2,X_1\cup X_2,X_1\cap X_2\in \mathcal{F}$, then 
\begin{align}
\label{clm:submo:eq1}
    f_\pi(X_1)+f_\pi(X_2)\geq f_\pi(X_1\cup X_2)+f_\pi(X_1\cap X_2)
\end{align}
holds.
\end{clm}
\begin{proof}
To show the submodular inequality, we first prepare the following equation:
\begin{align}
\label{clm:submo:eq2}
    |X_1\setminus S|+|X_2\setminus S|=|(X_1\cup X_2)\setminus S|+|(X_1\cap X_2)\setminus S|.
\end{align}
The following two equalities also hold:
\begin{itemize}
    \item $\pi(X_1\cap S)\cup \pi(X_2\cap S)=\pi((X_1\cup X_2)\cap S)$.
    \item $\pi(X_1\cap S)\cap \pi(X_2\cap S)\supseteq \pi((X_1\cap X_2)\cap S)$.
\end{itemize}
These equalities imply that 
\begin{align}
\label{clm:submo:eq3}
    |\pi(X_1\cap S)|+|\pi(X_2\cap S)|&=|\pi(X_1\cap S)\cup \pi(X_2\cap S)|+|\pi(X_1\cap S)\cap \pi(X_2\cap S)|\notag\\
    &\geq |\pi((X_1\cup X_2)\cap S)|+|\pi((X_1\cap X_2)\cap S)|.
\end{align}
By (\ref{clm:submo:eq2}) and (\ref{clm:submo:eq3}), we obtain (\ref{clm:submo:eq1}).
\end{proof}
A set $X\in \mathcal{F}$ is called \textit{$\pi$-satisfying} if $f_\pi(X)\geq g(X)$ holds. Moreover, $X$ is called \textit{$\pi$-tight} if $f_\pi(X)=g(X)$. Let $S'\subseteq U$ be a subset and $\pi':S'\rightarrow [k]$ a color assignment. Suppose that every $X\in \mathcal{F}$ is $\pi$-satisfying and $\pi'$-satisfying, where $f_{\pi'}$ is defined in a similar way to $f_{\pi}$. The following claim provides a sufficient condition to maintain the $\pi$-tightness when taking union.
\begin{clm}
\label{clm:tight}
Let $X_1,X_2,X_3\in \mathcal{F}$ be distinct sets satisfying the following three conditions:
\begin{itemize}
    \item $X_1\cap X_2\cap X_3\neq \emptyset$.
    \item $X_1$ and $X_2$ are $\pi$-tight.
    \item $X_1$ and $X_3$ are $\pi'$-tight.
\end{itemize}
Then at least one of the following two conditions holds:
\begin{itemize}
    \item $X_1\cup X_2\in \mathcal{F}$, and $X_1\cup X_2$ is $\pi$-tight. 
    \item $X_1\cup X_3\in \mathcal{F}$, and $X_1\cup X_3$ is $\pi'$-tight. 
\end{itemize}
\end{clm}
\begin{proof}
Since $X_1\cap X_2\cap X_3\neq \emptyset$, one of the following two conditions holds by intersecting 2/3-supermodularity of $g$:
\begin{enumerate}
\item $X_1\cup X_2,X_1\cap X_2\in \mathcal{F}\ \mathrm{and}\ g(X_1)+g(X_2)\leq g(X_1\cup X_2)+g(X_1\cap X_2)$.
\item $X_1\cup X_3,X_1\cap X_3\in \mathcal{F}\ \mathrm{and}\ g(X_1)+g(X_3)\leq g(X_1\cup X_3)+g(X_1\cap X_3)$.
\end{enumerate}
Consider the case when the condition 1 holds. Since $X_1$ and $X_2$ are $\pi$-tight, we have
\begin{align}
\label{clm:tight:eq2}
    f_\pi(X_1)+f_\pi(X_2)=g(X_1)+g(X_2)&\leq g(X_1\cup X_2)+g(X_1\cap X_2)\notag\\ &\leq f_\pi(X_1\cup X_2)+f_\pi(X_1\cap X_2)
\end{align}
(the last inequality holds because $X_1\cup X_2$ and $X_1\cap X_2$ are $\pi$-satisfying). By Claim \ref{clm:submo}, all of the inequalities in (\ref{clm:tight:eq2}) hold with equalities, which implies that $X_1\cup X_2$ is $\pi$-tight. Similarly, $X_1\cup X_3$ is $\pi'$-tight in the case when the condition 2 holds.
\end{proof}
We are now ready to prove Theorem \ref{thm:vizing_supermodular}.

\begin{proof}[Proof of Theorem \ref{thm:vizing_supermodular}]
Let $T_0\subseteq U$ be a maximum subset such that there exists a color assignment $\pi_0:T_0\rightarrow [k]$ satisfying
\begin{align}
\label{prfeq:visupermo}
    |X\setminus T_0|+|\pi_0(X\cap T_0)|\geq g(X)
\end{align}
for each $X\in \mathcal{F}$. Such a set $T_0$ does exist because $T_0=\emptyset$ satisfies (\ref{prfeq:visupermo}) for each $X\in \mathcal{F}$. Our aim is to show that $T_0=U$, which implies that (\ref{prfeq:visupermo}) coincides with (\ref{thmeq:visupermo}). Suppose for contradiction that $T_0\neq U$. Take $u_0\in U\setminus T_0$. Let $\mathcal{F}_0\subseteq \mathcal{F}$ be a family consisting of all $\pi_0$-tight sets $X\in \mathcal{F}$ with $u_0\in X$. Then, the number of maximal sets in $\mathcal{F}_0$ is at most two. Indeed, if distinct $X_1,X_2,X_3$ are maximal sets in $\mathcal{F}_0$, then one of $X_1\cup X_2$ and $X_1\cup X_3$ is a $\pi_0$-tight set in $\mathcal{F}$ by Claim \ref{clm:tight}, which contradicts the maximality of $X_1$. If the number of maximal sets in $\mathcal{F}_0$ is 0, i.e., $\mathcal{F}_0$ is the empty family, then extend the domain $T_0$ of $\pi_0$ to $T:=T_0\cup \{u_0\}$ and set $\pi_0(u_0):=\alpha$ for any color $\alpha$. For this extended color assignment $\pi_0$, each $X\in \mathcal{F}$ is $\pi_0$-satisfying, which contradicts the maximality of $T_0$. So we may assume that $\mathcal{F}_0$ is not empty. If $M$ is the unique maximal set in $\mathcal{F}_0$, then extend the domain $T_0$ of $\pi_0$ to $T$ and set $\pi_0(u_0):=\alpha$ for some $\alpha\notin \pi_0(M\cap T_0)$. For this extended color assignment $\pi_0$, each $X\in \mathcal{F}$ is $\pi_0$-satisfying because any $X\in \mathcal{F}_0$ satisfies $\alpha\notin \pi_0(M\cap T_0)\supseteq \pi_0(X\cap T_0)$, which contradicts the maximality of $T_0$. Hence, we may assume that $\mathcal{F}_0$ has the only two maximal sets $Y_0$ and $Z$. Suppose to the contrary that $Y_0,Z\in \mathcal{L}$. Then we have $Y_0\cup Z,Y_0\cap Z\in \mathcal{F}$ and $g(Y_0)+g(Z)\leq g(Y_0\cup Z)+g(Y_0\cap Z)$ because $Y_0$ and $Z$ are maximal sets in $\mathcal{F}_0$, and $u_0\in Y_0\cap Z$. Hence, we have
\begin{align}
\label{thm:visupermo:eq2}
    f_{\pi_0}(Y_0)+f_{\pi_0}(Z)=g(Y_0)+g(Z)\leq g(Y_0\cup Z)+g(Y_0\cap Z)\leq f_{\pi_0}(Y_0\cup Z)+f_{\pi_0}(Y_0\cap Z).
\end{align}
By Claim \ref{clm:submo}, all of the inequalities in (\ref{thm:visupermo:eq2}) hold with equality, which implies that $Y_0\cup Z$ is $\pi_0$-tight. Hence, we have $Y_0\cup Z\in \mathcal{F}_0$, which contradicts the maximality of $Y_0$. So we may assume without loss of generality that $Y_0\notin \mathcal{L}$. Let $\{(Y_0,u_0),\ldots,(Y_l,u_l)\}$ be a maximal sequence of pairs consisting of a set $Y_i\in \mathcal{F}$ and an element $u_i\in Z$ satisfying the following six conditions (see Figure \ref{fig:bicolorchain}):
\begin{enumerate}
    \item $u_0,\ldots,u_l$ are distinct.
    \item $u_1,\ldots,u_l\in T_0$.
    \item For each $i\in [l]$, define $\alpha_i=\pi_0(u_i)$, $T_i=T\setminus \{u_i\}$, and define $\pi_i:T_i\rightarrow [k]$ as follows:
    \begin{align}
    \pi_i(t)=\left\{
    \begin{array}{ll}
    \alpha_{j+1} & (t=u_j\ \mathrm{for\ some\ }0\leq j\leq i-1),\\
    \pi_0(t) & (\mathrm{otherwise}).
    \end{array}
    \right.\notag
    \end{align}
    Define $\mathcal{F}_i$ as a family consisting of all $\pi_i$-tight sets $X\in \mathcal{F}$ with $u_i\in X$ for each $i\in [l]$. Then $Y_i$ is a maximal set in $\mathcal{F}_i$ distinct from $Z$ for each $0\leq i\leq l$.
    \item $\alpha_{i+1}\notin \pi_0(Y_i\cap T_0)$ for each $0\leq i\leq l-1$.
    \item $\alpha_{i+1}\neq \alpha_{j+1}$ for each $0\leq i<j\leq l-1$ with $u_i\in Y_j$.
    \item $Y_0,\ldots,Y_l\notin \mathcal{L}$.
\end{enumerate}
\begin{figure}[tb]
  \centering
  \includegraphics[width=8cm]{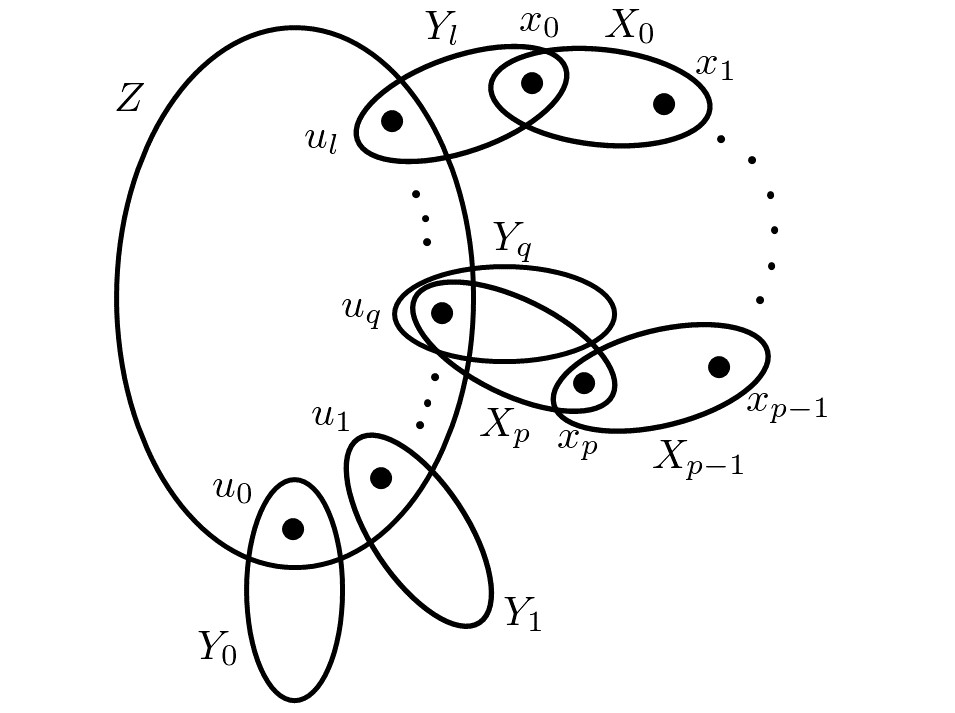}
  \caption{A maximal sequence $\{(Y_0,u_0),\ldots,(Y_l,u_l)\}$ and a maximal sequence $\{x_0,\ldots,x_p\}$.}
  \label{fig:bicolorchain}
\end{figure}
Such a maximal sequence does exist because a sequence $\{(Y_0,u_0)\}$ satisfies all of the above conditions. We have the following facts concerning this maximal sequence.
\begin{clm}
\label{clm:z_in_f}
$Z$ is a maximal set in $\mathcal{F}_i$ for each $i\in [l]$.
\end{clm}
\begin{proof}
We first show that $Z\in \mathcal{F}_i$ for each $i\in [l]$. Since $u_0,\ldots,u_{i}\in Z$, we have $|Z\setminus T_i|=|Z\setminus T_0|$ and $|\pi_i(Z\cap T_i)|=|\pi_0(Z\cap T_0)|$, which implies that $Z$ is $\pi_i$-tight because $Z$ is $\pi_0$-tight. Combined with $u_i\in Z$, this implies that $Z\in \mathcal{F}_i$. Suppose to the contrary that $Z$ is not a maximal set in $\mathcal{F}_i$ for some $i\in [l]$. Then there exists a maximal set $M_i\supsetneq Z$ in $\mathcal{F}_i$. Since $u_0,\ldots,u_{i}\in M_i$, and since $M_i$ is $\pi_i$-tight, $M_i$ is also $\pi_0$-tight, which contradicts that $Z$ is a maximal set in $\mathcal{F}_0$.
\end{proof}
\begin{clm}
\label{clm:pi_i-sat}
Every $X\in \mathcal{F}$ is $\pi_i$-satisfying for each $0\leq i\leq l$.
\end{clm}
\begin{proof}
We prove the claim by induction on $i$. The case $i=0$ is trivial. Suppose that every $X\in \mathcal{F}$ is $\pi_i$-satisfying for some $0\leq i\leq l-1$. Then it suffices to show that every $X\in \mathcal{F}$ is $\pi_{i+1}$-satisfying. Take $X\in \mathcal{F}$. We now consider separately the following four cases depending on whether $u_i\in X$ or not, and whether $u_{i+1}\in X$ or not:
\begin{itemize}
    \item Consider the case when $u_i,u_{i+1}\notin X$. Then we have $|X\setminus T_i|=|X\setminus T_{i+1}|$ and $|\pi_i(X\cap T_i)|=|\pi_{i+1}(X\cap T_{i+1})|$, which implies that $X$ is $\pi_{i+1}$-satisfying because $X$ is $\pi_{i}$-satisfying.
    \item Consider the case when $u_i,u_{i+1}\in X$. Then we have $|X\setminus T_i|=|X\setminus T_{i+1}|$ and $|\pi_i(X\cap T_i)|=|\pi_{i+1}(X\cap T_{i+1})|$, which implies that $X$ is $\pi_{i+1}$-satisfying because $X$ is $\pi_{i}$-satisfying.
    \item Consider the case when $u_i\notin X$ and $u_{i+1}\in X$. Then we have $|X\setminus T_i|+1=|X\setminus T_{i+1}|$ and $|\pi_i(X\cap T_i)|\leq |\pi_{i+1}(X\cap T_{i+1})|+1$, which implies that $X$ is $\pi_{i+1}$-satisfying because $X$ is $\pi_{i}$-satisfying.
    \item Consider the case when $u_i\in X$ and $u_{i+1}\notin X$. Then we have $|X\setminus T_i|-1=|X\setminus T_{i+1}|$ and $|\pi_i(X\cap T_i)|\leq |\pi_{i+1}(X\cap T_{i+1})|$. Hence, if $X$ is not $\pi_i$-tight, then $X$ is $\pi_{i+1}$-satisfying. Consider the case when $X$ is $\pi_i$-tight. If $X=Z$, then $u_{i+1}\in Z$ contradicts $u_{i+1}\notin X$. If $X=Y_i$, then since $\alpha_{i+1}\notin \pi_0(Y_i\cap T_0)$ and $\alpha_{j+1}\neq \alpha_{i+1}$ for every $0\leq j<i$ with $u_j\in Y_i$, we have $\alpha_{i+1}\notin \pi_i(Y_i\cap T_i)=\pi_i(X\cap T_i)$, which implies that $|\pi_i(X\cap T_i)|+1=|\pi_{i+1}(X\cap T_{i+1})|$, and hence $X$ is $\pi_{i+1}$-satisfying. Consider the case when $X\neq Z,Y_i$. By Claim \ref{clm:z_in_f}, $Z$ is $\pi_i$-tight. Then by Claim \ref{clm:tight}, one of $Y_i\cup X$ and $Y_i\cup Z$ is a $\pi_i$-tight set in $\mathcal{F}$. Since $Y_i$ and $Z$ are distinct maximal sets in $\mathcal{F}_i$ (by Claim \ref{clm:z_in_f}), $Y_i\cup Z$ is not a $\pi_i$-tight set in $\mathcal{F}$, which implies that $Y_i\cup X$ is a $\pi_i$-tight set in $\mathcal{F}$. Hence, we have $Y_i\supseteq Y_i\cup X$ by the maximality of $Y_i$, which implies that $Y_i\supseteq X$. Thus, we have $\alpha_{i+1}\notin \pi_i(Y_i\cap T_i)\supseteq \pi_i(X\cap T_i)$, which implies that $X$ is $\pi_{i+1}$-satisfying.
\end{itemize}
\end{proof}
\begin{clm}
\label{clm:maximal_in_f_i}
$Z$ and $Y_i$ are the only maximal sets in $\mathcal{F}_i$ for each $i\in [l]$.
\end{clm}
\begin{proof}
If there exists a maximal set $X\neq Z,Y_i$ in $\mathcal{F}_i$, then one of $Z\cup Y_i$ and $Z\cup X$ is a $\pi_i$-tight set in $\mathcal{F}$ by Claim \ref{clm:tight} and Claim \ref{clm:pi_i-sat}, which contradicts the maximality of $Z$ in $\mathcal{F}_i$.
\end{proof}
\begin{clm}
\label{clm:y_i=y_j}
$Y_i=Y_j$ holds for every $0\leq i\neq j\leq l$ with $u_i\in Y_j$.
\end{clm}
\begin{proof}
Suppose to the contrary that $Y_i\neq Y_j$ for some $0\leq i\neq j\leq l$ with $u_i\in Y_j$. Then by Claim \ref{clm:tight}, Claim \ref{clm:pi_i-sat}, and Claim \ref{clm:maximal_in_f_i}, we have $Z\cup Y_i\in \mathcal{F}_i$ or $Z\cup Y_j\in \mathcal{F}_j$, which contradicts the maximality of $Z$.
\end{proof}
\begin{clm}
\label{clm:color}
There exists a color $\alpha_{l+1}\notin \pi_0(Y_l\cap T_0)$ satisfying $\alpha_{l+1}\neq \alpha_{i+1}$ for every $0\leq i<l$ with $u_i\in Y_l$.
\end{clm}
\begin{proof}
Let $\tilde{Y}_l=\{u_i\in Y_l\mid 0\leq i\leq l\}$. It suffices to show that $|\pi_0(Y_l\cap T_0)|+|\tilde{Y}_l|\leq k$. By Claim \ref{clm:y_i=y_j}, we have $\alpha_{i+1}\notin \pi_0(Y_i\cap T_0)=\pi_0(Y_l\cap T_0)$ for each $0\leq i<l$ with $u_i\in Y_l$, and we have $\alpha_{i+1}\neq \alpha_{j+1}$ for every $0\leq i\neq j<l$ with $u_i,u_j\in Y_l$, which implies that 
\begin{align}
\label{clm:color:eq1}
    |\pi_l(Y_l\cap T_l)|&=|\pi_l((Y_l\setminus \tilde{Y}_l)\cap T_l)|+|\pi_l(\tilde{Y}_l\cap T_l)|=|\pi_0((Y_l\setminus \tilde{Y}_l)\cap T_0)|+|\tilde{Y}_l\cap T_l|\notag\\&\geq |\pi_0((Y_l\setminus \tilde{Y}_l)\cap T_0)|+|\tilde{Y}_l\cap T_0|-1\geq |\pi_0(Y_l\cap T_0)|-1
\end{align}
Since $Y_l$ is $\pi_l$-tight, we have
\begin{align}
\label{clm:color:eq2}
    |Y_l\setminus T_l|+|\pi_l(Y_l\cap T_l)|=g(Y_l)
\end{align}
Since $Y_l\notin \mathcal{L}$, we have
\begin{align}
\label{clm:color:eq3}
    g(Y_l)+D_{\mathcal{F}}(Y_l)\leq k.
\end{align}
Since $u_l\in Y_l\setminus T_l$, we have
\begin{align}
\label{clm:color:eq4}
    |\tilde{Y}_l|\leq |Y_l\cap Z|\leq D_{\mathcal{F}}(Y_l)\leq D_{\mathcal{F}}(Y_l)+|Y_l\setminus T_l|-1.
\end{align}
By (\ref{clm:color:eq1}), (\ref{clm:color:eq2}), (\ref{clm:color:eq3}), and (\ref{clm:color:eq4}), we have
\begin{align}
    |\pi_0(Y_l\cap T_0)|+|\tilde{Y}_l|&\leq |\pi_0(Y_l\cap T_0)|+D_{\mathcal{F}}(Y_l)+|Y_l\setminus T_l|-1\notag\\&\leq |\pi_l(Y_l\cap T_l)|+D_{\mathcal{F}}(Y_l)+|Y_l\setminus T_l|=g(Y_l)+D_{\mathcal{F}}(Y_l)\leq k.\notag
\end{align}
\end{proof}

Take a color $\alpha_{l+1}$ in Claim \ref{clm:color}. If $\alpha_{l+1}\notin \pi_0(Z\cap T_0)$, then extend the domain $T_l$ of $\pi_l$ to $T=T_l\cup \{u_l\}$ and set $\pi_l(u_l):=\alpha_{l+1}$. For this extended color assignment $\pi_l$, each $X\in \mathcal{F}$ is $\pi_l$-satisfying because the only maximal sets $Z$ and $Y_l$ in $\mathcal{F}_l$ satisfy $\alpha_{l+1}\notin \pi_0(Z\cap T_0)=\pi_l(Z\cap T_l)$ and $\alpha_{l+1}\notin \pi_l(Y_l\cap T_l)$, which contradicts the maximality of $T_0$. So we may assume that $\alpha_{l+1}\in \pi_0(Z\cap T_0)$. 

Consider the case when there exists an element $u_{l+1}\in Z\cap T_0$ distinct from $u_0,\ldots,u_l$ and satisfying $\pi_0(u_{l+1})=\alpha_{l+1}$. Let $T_{l+1}=T\setminus \{u_{l+1}\}$. Define the following color assignment $\pi_{l+1}:T_{l+1}\rightarrow [k]$:
\begin{align}
    \pi_{l+1}(t)=\left\{
    \begin{array}{ll}
    \alpha_{i+1} & (t=u_i\ \mathrm{for\ some\ }0\leq i\leq l),\\
    \pi_0(t) & (\mathrm{otherwise}).
    \end{array}
    \right.\notag
    \end{align}
By a similar argument to the proof of Claim \ref{clm:pi_i-sat}, we see that every $X\in \mathcal{F}$ is $\pi_{l+1}$-satisfying. Let $\mathcal{F}_{l+1}\subseteq \mathcal{F}$ be a family consisting of all $\pi_{l+1}$-tight sets $X\in \mathcal{F}$ with $u_{l+1}\in X$. By a similar argument to the proof of Claim \ref{clm:z_in_f}, we see that $Z$ is a maximal set in $\mathcal{F}_{l+1}$. If $Z$ is the unique maximal set in $\mathcal{F}_{l+1}$, then extend the domain $T_{l+1}$ of $\pi_{l+1}$ to $T$ and set $\pi_{l+1}(u_{l+1})=\alpha$ with $\alpha\notin \pi_0(Z\cap T_0)$. For this extended color assignment $\pi_{l+1}$, every set $X\in \mathcal{F}$ is $\pi_{l+1}$-satisfying, which contradicts the maximality of $T_0$. Hence, we may assume that there exists a maximal set $Y_{l+1}\neq Z$ in $\mathcal{F}_{l+1}$. By the maximality of the sequence $\{(Y_0,u_0),\ldots,(Y_l,u_l)\}$, we have $Y_{l+1}\in \mathcal{L}$. If $Z\in \mathcal{L}$, then since every $X\in \mathcal{F}$ is $\pi_{l+1}$-satisfying, and since $Y_{l+1}$ and $Z$ are $\pi_{l+1}$-tight, and since $\mathcal{L}$ is a $g$-laminar family, we have $Y_{l+1}\cup Z,Y_{l+1}\cap Z\in \mathcal{F}$ and 
\begin{align}
\label{thm:visupermo:eq3}
    f_{\pi_{l+1}}(Y_{l+1})+f_{\pi_{l+1}}(Z)&=g(Y_{l+1})+g(Z)\leq g(Y_{l+1}\cup Z)+g(Y_{l+1}\cap Z)\notag\\&\leq f_{\pi_{l+1}}(Y_{l+1}\cup Z)+f_{\pi_{l+1}}(Y_{l+1}\cap Z).
\end{align}
By Claim \ref{clm:submo}, all of the inequalities in (\ref{thm:visupermo:eq3}) hold with equalities, which implies that $Y_{l+1}\cup Z\in \mathcal{F}_{l+1}$, a contradiction. Hence, we may assume that $Z\notin \mathcal{L}$. We now redefine $Z,Y_0,u_0,T_0,\pi_0,\mathcal{F}_0$ to be $Y_{l+1},Z,u_{l+1},T_{l+1},\pi_{l+1},\mathcal{F}_{l+1}$, respectively. Then we again start from the beginning of this proof with redefined $Z,Y_0,u_0,T_0,\pi_0,\mathcal{F}_0$. Since redefined $Z$ satisfies $Z\in \mathcal{L}$, we may assume that $Y_{l+1}\notin \mathcal{L}$ in the redefined setting, which contradicts the maximality of the sequence $\{(Y_0,u_0),\ldots,(Y_l,u_l)\}$. Hence, we may assume that there exist no elements $u_{l+1}\in Z\cap T_0$ distinct from $u_0,\ldots,u_l$ and satisfying $\pi_0(u_{l+1})=\alpha_{l+1}$.

Take a color $\beta\notin \pi_0(Z\cap T_0)$. If $\beta\notin \pi_l(Y_l\cap T_l)$, then extend the domain $T_{l}$ of $\pi_{l}$ to $T$ and set $\pi_{l}(u_{l}):=\beta$. For this extended color assignment $\pi_{l}$, each $X\in \mathcal{F}$ is $\pi_{l}$-satisfying because of Claim~\ref{clm:maximal_in_f_i}, which contradicts the maximality of $T_0$. Hence, we may assume that $\beta\in \pi_l(Y_l\cap T_l)$. Let $x_0\in Y_l$ be an element with $\pi_l(x_0)=\beta$. 
We now construct a ``bicolor chain'' starting at $x_0$ and consisting of elements with colors $\beta$ and $\alpha_{l+1}$. 
Let $\{x_0,\ldots,x_p\}$ be a maximal sequence consisting of distinct elements in $U\setminus Z$ and satisfying the following three conditions (see Figure \ref{fig:bicolorchain}):
\begin{enumerate}
    \item For each $0\leq i\leq p$, $\pi_l(x_i)=\beta$ if $i$ is even, and $\pi_l(x_i)=\alpha_{l+1}$ if $i$ is odd.
    \item For each $0\leq i\leq p$, define the following color assignment $\pi_l^i:T_l\rightarrow [k]$:
    \begin{align}
    \pi_l^i(t)=\left\{
    \begin{array}{lll}
    \alpha_{l+1} & (t=x_j\ \mathrm{for\ even\ }0\leq j\leq i),\\
    \beta & (t=x_j\ \mathrm{for\ odd\ }0\leq j\leq i),\\
    \pi_l(t) & (\mathrm{otherwise}).
    \end{array}
    \right.\notag
    \end{align}
    Then there exists the unique $X_i\in \mathcal{F}$ which is not $\pi_l^i$-satisfying for each $0\leq i\leq p-1$.
    \item $x_{i+1}\in X_i$ for each $0\leq i\leq p-1$.
\end{enumerate}
Such a sequence does exist because the sequence $\{x_0\}$ satisfies the above conditions.
\begin{clm}
\label{clm:x_iinx_i}
$x_{i}\in X_i$ holds for every $0\leq i\leq p-1$.
\end{clm}
\begin{proof}
If $x_0\notin X_0$, then $\pi_l^0(X_0\cap T_l)=\pi_l(X_0\cap T_l)$, which implies that $X_0$ is $\pi_l^0$-satisfying, a contradiction. Otherwise, we have $x_0\in X_0$. Suppose to the contrary that $x_i\notin X_i$ for some $1\leq i\leq p-1$. Then $X_i$ is not $\pi_l^{i-1}$-satisfying because $X_i$ is not $\pi_l^{i}$-satisfying. This implies that $X_i=X_{i-1}$. Hence we have $x_{i}\in X_{i-1}=X_i$, which contradicts $x_i\notin X_i$.
\end{proof}
Let $\tilde{\pi}_l^p:T\rightarrow [k]$ be an assignment of colors defined below:
\begin{align}
    \tilde{\pi}_l^p(t)=\left\{
    \begin{array}{ll}
    \beta & (t=u_l),\\
    \pi_l^p(t) & (\mathrm{otherwise}).
    \end{array}
    \right.\notag
    \end{align}
If every $X\in \mathcal{F}$ is $\pi_l^p$-satisfying, then we can derive a contradiction by the following fact.
\begin{clm}
\label{clm:tildepil_i}
If every $X\in \mathcal{F}$ is $\pi_l^p$-satisfying, then every $X\in \mathcal{F}$ is also $\tilde{\pi}_l^p$-satisfying.
\end{clm}
\begin{proof}
Suppose for contradiction that every $X\in \mathcal{F}$ is $\pi_l^p$-satisfying, and that some $\tilde{X}_p\in \mathcal{F}$ is not $\tilde{\pi}_l^p$-satisfying. Since $Z$ is $\pi_l^p$-satisfying, and since $\beta\notin \pi_l(Z\cap T_l)=\pi_l^p(Z\cap T_l)$, $Z$ is $\tilde{\pi}_l^p$-satisfying. Since $Y_l$ is $\pi_l$-satisfying, and since $\tilde{\pi}_l^p(Y_l\cap T)=\pi_l(Y_l\cap T_l)\cup \{\alpha_{l+1}\}$, $Y_l$ is $\tilde{\pi}_l^p$-satisfying. Hence, we have $\tilde{X}_p\neq Z,Y_l$. Since $Z$ is $\pi_l$-tight (by Claim \ref{clm:z_in_f}), $Z$ is also $\pi_l^p$-tight. Since $\tilde{X}_p$ is $\pi_l^p$-satisfying and not $\tilde{\pi}_l^p$-satisfying, $\tilde{X}_p$ is $\pi_l^p$-tight, and $u_l\in \tilde{X}_p$, and $\beta\in \pi_l^p(\tilde{X}_p\cap T_l)$. Then, by Claim \ref{clm:tight} one of the following conditions holds:
\begin{itemize}
    \item $Z\cup Y_l\in \mathcal{F}$, and $Z\cup Y_l$ is $\pi_l$-tight.
    \item $Z\cup \tilde{X}_p\in \mathcal{F}$, and $Z\cup \tilde{X}_p$ is $\pi_l^p$-tight.
\end{itemize}
As the former contradicts the maximality of $Y_l$ by Claim \ref{clm:maximal_in_f_i}, we may assume that the latter holds. Since $\alpha_{l+1}\in \pi_l(Z\cap T_l)=\pi_l^p(Z\cap T_l)$, we have $\alpha_{l+1},\beta\in \pi_l^p((Z\cup \tilde{X}_p)\cap T_l)$. Hence, we have $\pi_l((Z\cup \tilde{X}_p)\cap T_l)\subseteq \pi_l^p((Z\cup \tilde{X}_p)\cap T_l)$. Since $Z\cup \tilde{X}_p$ is $\pi_l^p$-tight, this implies that $Z\cup \tilde{X}_p$ is $\pi_l$-tight and that $\alpha_{l+1},\beta\in \pi_l((Z\cup \tilde{X}_p)\cap T_l)$. Hence, we have $Z\subsetneq Z\cup \tilde{X}_p$ because $\beta\notin \pi_l(Z\cap T_l)$, contradicting the maximality of $Z$ in $\mathcal{F}_l$.
\end{proof}
By Claim \ref{clm:tildepil_i}, it suffices to consider the case when some $X\in \mathcal{F}$ is not $\pi_l^p$-satisfying.
\begin{clm}
\label{clm:notsat1}
The number of $X\in \mathcal{F}$ which is not $\pi_l^p$-satisfying is at most one.
\end{clm}
\begin{proof}
Suppose for contradiction that distinct $W_1,W_2\in \mathcal{F}$ are not $\pi_l^p$-satisfying. Then $W_1$ and $W_2$ are $\pi_l$-tight. Since $Y_l$ is $\pi_l^p$-satisfying, we have $Y_l\neq W_1,W_2$. Consider the case when $p=0$. In this case, since $W_1$ and $W_2$ are $\pi_l$-satisfying and not $\pi_l^0$-satisfying, we have $x_0\in W_1,W_2$ and $\alpha_{l+1}\in \pi_l(W_1\cap T_l),\pi_l(W_2\cap T_l)$. Then, by Claim \ref{clm:tight} one of $Y_l\cup W_1$ and $Y_l\cup W_2$ is a $\pi_l$-tight set in $\mathcal{F}$. Since $\alpha_{l+1}\in \pi_l(W_1\cap T_l)\setminus \pi_l(Y_l\cap T_l),\pi_l(W_2\cap T_l)\setminus \pi_l(Y_l\cap T_l)$ implies $Y_l\not\supseteq W_1,W_2$, this contradicts the maximality of $Y_l$. 

Consider the case when $p\geq 1$. If $W_1$ is not $\pi_l^{p-1}$-satisfying, then $W_1=X_{p-1}$, which implies that $x_{p-1},x_p\in W_1$ by Claim \ref{clm:x_iinx_i}. Since $W_1$ is $\pi_l$-satisfying, $W_1$ is also $\pi_l^p$-satisfying, a contradiction. Hence, we may assume that $W_1$ and $W_2$ are $\pi_l^{p-1}$-satisfying. This implies that $x_p\in W_1\cap W_2$ because $W_1$ and $W_2$ are not $\pi_l^{p}$-satisfying. Since $X_{p-1},W_1,W_2$ are $\pi_l$-tight, and since $x_p\in X_{p-1}\cap W_1\cap W_2$, one of $X_{p-1}\cup W_1$ and $X_{p-1}\cup W_2$ is a $\pi_l$-tight set in $\mathcal{F}$ by Claim \ref{clm:tight}. We may assume without loss of generality that $X_{p-1}\cup W_1$ is a $\pi_l$-tight set in $\mathcal{F}$. Let $S_1=X_{p-1}\cup W_1$. Since $W_1$ is $\pi_l^{p-1}$-satisfying and not $\pi_l^p$-satisfying, and since $x_p\in W_1$, we have $x_{p-1}\notin W_1$ and $\pi_l^p(x_p)\in \pi_l^p((W_1\setminus \{x_p\})\cap T_l)$. Take $w\in W_1$ such that $w\neq x_{p-1},x_{p}$ and $\pi_l^p(w)=\pi_l^p(x_p)$. If $W_1\subseteq X_{p-1}$, then we have $w\in X_{p-1}$, which contradicts that $X_{p-1}$ is not $\pi_l^{p-1}$-satisfying. Otherwise, we have $W_1\not\subseteq X_{p-1}$. This implies that $S_1\supsetneq X_{p-1}$. If $X_{p-1}\subseteq X_{p-2}$ or $X_{p-2}\subseteq X_{p-1}$ holds, then we have $x_p\in X_{p-2}$ or $x_{p-2}\in X_{p-1}$, which contradicts that $X_{p-2}$ is not $\pi_l^{p-2}$-satisfying, or contradicts that $X_{p-1}$ is not $\pi_l^{p-1}$-satisfying. Otherwise, we have $X_{p-1}\not\subseteq X_{p-2}\not\subseteq X_{p-1}$. Then $S_1,X_{p-1},X_{p-2}$ are distinct $\pi_l$-tight sets. Hence, one of $X_{p-2}\cup S_1$ and $X_{p-2}\cup X_{p-1}$ is a $\pi_l$-tight set in $\mathcal{F}$ by Claim \ref{clm:tight}. Let $S_2$ be this $\pi_l$-tight set. Then $S_2\supsetneq X_{p-2}$ holds because of $X_{p-1}\not\subseteq X_{p-2}$. We also have $X_{p-2}\not\subseteq X_{p-3}\not\subseteq X_{p-2}$ by a similar argument as above. Then $S_2,X_{p-2},X_{p-3}$ are distinct $\pi_l$-tight sets. By Claim \ref{clm:tight}, this implies that one of $X_{p-3}\cup S_2$ and $X_{p-3}\cup X_{p-2}$ is a $\pi_l$-tight set in $\mathcal{F}$. Let $S_3$ be this $\pi_l$-tight set. Then $S_3\supsetneq X_{p-3}$ holds because of $X_{p-2}\not\subseteq X_{p-3}$. By repeating the similar arguments, we finally obtain a $\pi_l$-tight set $S_p$ in $\mathcal{F}$ satisfying $S_p\supsetneq X_0$. By Claim \ref{clm:x_iinx_i}, we have $x_0\in S_p\cap X_0\cap Y_l$. Since $x_1\in X_0$ and $\pi_l(x_1)=\alpha_{l+1}$, we have $x_1\in X_0\setminus Y_l$, which implies that $Y_l\not\supseteq X_0$. By these facts, $S_p,X_0,Y_l$ are distinct $\pi_l$-tight sets satisfying $S_p\cap X_0\cap Y_l\neq \emptyset$. By Claim \ref{clm:tight}, this implies that one of $Y_l\cup S_p$ and $Y_l\cup X_0$ is a $\pi_l$-tight set in $\mathcal{F}$, contradicting the maximality of $Y_l$.
\end{proof}
By Claim \ref{clm:notsat1}, it suffices to consider the case when there exists the unique set $X_p\in \mathcal{F}$ which is not $\pi_l^p$-satisfying. By the proof of Claim \ref{clm:notsat1}, we have $x_p\in X_p$. Since $X_p$ is $\pi_l$-satisfying and not $\pi_l^p$-satisfying, we have $\alpha_{l+1},\beta\in \pi_l(X_p\cap T_l)$ and $\{\alpha_{l+1},\beta\}\not\subseteq \pi_l^p(X_p\cap T_l)$. Hence, there exists an element $x_{p+1}\in X_p\cap T_l$ such that $\pi_l(x_{p+1})\in \{\alpha_{l+1},\beta\}$ and $\pi_l(x_{p+1})\neq \pi_l(x_p)$. If $x_{p+1}=x_i$ holds for some $0\leq i\leq p-1$, then we have $\alpha_{l+1},\beta\in \pi_l^p(X_p\cap T_l)$, a contradiction. Otherwise, $x_{p+1}\neq x_0,\ldots,x_{p-1}$. Then, by the maximality of the sequence $\{x_0,\ldots,x_p\}$, we have $x_{p+1}\in Z$. Since $\beta\notin \pi_l^p(Z\cap T_l)$, this implies that $\pi_l(x_{p+1})=\pi_l^p(x_{p+1})=\alpha_{l+1}$. Recall that there exist no elements $t\in Z\cap T_0$ distinct from $u_0,\ldots,u_l$ and satisfying $\pi_0(t)=\alpha_{l+1}$. If $x_{p+1}\neq u_0,\ldots,u_{l-1}$, then we have $\pi_0(x_{p+1})=\pi_l(x_{p+1})=\alpha_{l+1}$, a contradiction. Otherwise, we have $x_{p+1}=u_q$ for some $0\leq q\leq l-1$. Let $\pi_l^{p+1}:T_l\rightarrow [k]$ be an assignment of colors defined below:
\begin{align}
    \pi_l^{p+1}(t)=\left\{
    \begin{array}{ll}
    \alpha_{l+1} & (t=x_i\ \mathrm{for\ even\ }0\leq i\leq p+1),\\
    \beta & (t=x_i\ \mathrm{for\ odd\ }0\leq i\leq p+1),\\
    \pi_l(t) & (\mathrm{otherwise}).
    \end{array}
    \right.\notag
    \end{align}
\begin{clm}
\label{clm:pilp1-sat}
Every $X\in \mathcal{F}$ is $\pi_l^{p+1}$-satisfying.
\end{clm}
\begin{proof}
Suppose for contradiction that some $X\in \mathcal{F}$ is not $\pi_l^{p+1}$-satisfying. Then $X$ is $\pi_l$-tight because $X$ is $\pi_l$-satisfying. If $x_{p+1}\notin X$, then $X$ is not $\pi_l^p$-satisfying because $X$ is is not $\pi_l^{p+1}$-satisfying, which implies that $X=X_p$, contradicting $x_{p+1}\notin X$. Otherwise, we have $x_{p+1}\in X$. Then $x_{p+1}\in X\cap X_p\cap Z$. If $X\subseteq Z$, then $X$ is $\pi_l^{p}$-satisfying and we have $\beta\notin \pi_l^p(X\cap T_l)$, which implies that $X$ is $\pi_l^{p+1}$-satisfying, a contradiction. Otherwise, we have $X\not\subseteq Z$. Since $\beta\in \pi_l(X_p\cap T_l)$, we have $X_p\not\subseteq Z$. If $X=X_p$, then we have $x_p,x_{p+1}\in X$, which contradicts that $X$ is not $\pi_l^{p+1}$-satisfying. Otherwise, we have $X\neq X_p$. By these facts, $X,X_p,Z$ are distinct $\pi_l$-tight sets. Hence, by Claim \ref{clm:tight} one of $Z\cup X$ and $Z\cup X_p$ is a $\pi_l$-tight set in $\mathcal{F}$, contradicting the maximality of $Z$ in $\mathcal{F}_l$.
\end{proof}
For each $q+1\leq i\leq l-1$, let $\pi_i^{p+1}:T_i\rightarrow [k]$ be an assignment of colors defined below:
\begin{align}
    \pi_i^{p+1}(t)=\left\{
    \begin{array}{lll}
    \alpha_{l+1} & (t=x_j\ \mathrm{for\ even\ }0\leq j\leq p+1),\\
    \beta & (t=x_j\ \mathrm{for\ odd\ }0\leq j\leq p+1),\\
    \pi_i(t) & (\mathrm{otherwise}).
    \end{array}
    \right.\notag
    \end{align}
Let $\pi_q^{p+1}:T\rightarrow [k]$ be an assignment of colors defined below:
\begin{align}
    \pi_q^{p+1}(t)=\left\{
    \begin{array}{lll}
    \alpha_{q+1}=\alpha_{l+1} & (t=u_{q+1}),\\
    \pi_{q+1}^{p+1}(t) & (\mathrm{otherwise}).
    \end{array}
    \right.\notag
    \end{align}
The following fact completes the proof.
\begin{clm}
\label{clm:final}
Every $X\in \mathcal{F}$ is $\pi_i^{p+1}$-satisfying for each $q\leq i\leq l$.
\end{clm}
\begin{proof}
We show this by induction on $i$. The case of $i=l$ follows from Claim \ref{clm:pilp1-sat}. Suppose that every $X\in \mathcal{F}$ is $\pi_i^{p+1}$-satisfying for some $q+1\leq i\leq l$. It suffices to show that every $X\in \mathcal{F}$ is $\pi_{i-1}^{p+1}$-satisfying. Suppose for contradiction that some $X\in \mathcal{F}$ is not $\pi_{i-1}^{p+1}$-satisfying. Since $X$ is $\pi_i^{p+1}$-satisfying, we have $u_i\in X$. We now show that $X$ is $\pi_{i-1}$-tight. We consider two cases $q+2\leq i\leq l$ and $i=q+1$ separately.
\begin{itemize}
    \item Consider the case when $q+2\leq i\leq l$. In this case, since $X$ is $\pi_{i-1}$-satisfying and not $\pi_{i-1}^{p+1}$-satisfying, $X$ is $\pi_{i-1}$-tight.
    \item Consider the case when $i=q+1$. If $u_q\notin X$, then since $X$ is $\pi_{q}$-satisfying and not $\pi_{q}^{p+1}$-satisfying, $X$ is $\pi_{q}$-tight. If $u_q\in X$, then since $\pi_q^{p+1}(u_q)=\beta,\pi_q^{p+1}(u_{q+1})=\alpha_{l+1}$, we have $\beta,\alpha_{l+1}\in \pi_q^{p+1}(X\cap T)$. Since $X$ is $\pi_{q}$-satisfying and not $\pi_{q}^{p+1}$-satisfying, this implies that $X$ is $\pi_q$-tight.
\end{itemize}
If $X\subseteq Z$, then since $X$ is $\pi_{i-1}$-satisfying, $X$ is also $\pi_{i-1}^{p+1}$-satisfying because $x_0,\ldots,x_p\notin Z\supseteq X$ and $\pi_{i-1}^{p+1}(x_{p+1})=\beta\notin \pi_{i-1}(Z\cap T_{i-1})\supseteq \pi_{i-1}(X\cap T_{i-1})$, a contradiction. Hence, we may assume that $X\not\subseteq Z$. If $X\neq Y_i$, then since  $Z,Y_i,X$ are distinct sets satisfying $u_i\in Z\cap Y_i\cap X$, by Claim \ref{clm:tight} $Z\cup Y_i$ is a $\pi_i$-tight set in $\mathcal{F}$, or $Z\cup X$ is a $\pi_{i-1}$-tight set in $\mathcal{F}$, contradicting the maximality of $Z$. Hence, we may assume that $X=Y_i$. If $u_{i-1}\in Y_i$, then by Claim \ref{clm:y_i=y_j} we have $Y_{i-1}=Y_{i}$, which implies that $\alpha_{i}\notin \pi_{i-1}(Y_{i-1}\cap T_{i-1})=\pi_{i-1}(Y_i\cap T_{i-1})$, contradicting $\pi_{i-1}(u_i)=\alpha_i$. Hence, we may assume that $u_{i-1}\notin Y_i$. Since $Y_i$ is $\pi_i$-tight and $\pi_{i-1}$-tight, this implies that $\alpha_{i}\notin \pi_i(Y_i\cap T_i)$. Moreover, since $Y_i$ is $\pi_i^{p+1}$-satisfying and not $\pi_{i-1}^{p+1}$-satisfying, and since $u_{i-1}\notin Y_i$ and $u_{i}\in Y_i$, we have $\alpha_{i}\in \pi_i^{p+1}(Y_i\cap T_i)$. Since $\alpha_{i}\notin \pi_i(Y_i\cap T_i)$ and $\alpha_{i}\in \pi_i^{p+1}(Y_i\cap T_i)$, we have $\alpha_i \in \{\alpha_{l+1},\beta\}$. Since $\alpha_i \in \pi_i(Z\cap T_i)$ and $\beta\notin \pi_i(Z\cap T_i)$, this implies that $\alpha_i=\alpha_{l+1}$ and $x_s\in Y_i$ for some $0\leq s\leq p$. If $i=l$, then we have $\alpha_{l+1}=\alpha_{l}\in \pi_{l-1}(Y_l\cap T_{l-1})$, contradicting the definition of $\alpha_{l+1}$. Assume that $q+1\leq i\leq l-1$. We consider two cases $s=0$ and $1\leq s\leq p$ separately.

Consider the case when $s=0$. Since $\alpha_{l+1}\in \pi_{i-1}(Y_i\cap T_{i-1})$ and $\alpha_{l+1}\notin \pi_{i-1}(Y_l\cap T_{i-1})$, we have $Y_i\neq Y_l$. If $Y_l\subseteq Y_i$, then we have $u_l\in Z\cap Y_l\cap Y_i$, which implies by Claim \ref{clm:tight} that $Z\cup Y_l$ is a $\pi_l$-tight set in $\mathcal{F}$, or $Z\cup Y_i$ is a $\pi_i$-tight set in $\mathcal{F}$, contradicting the maximality of $Z$. Hence, we may assume that $Y_l\not\subseteq Y_i$. Since $\alpha_{l+1}\in \pi_l(X_0\cap T_l)$ and $\alpha_{l+1}\notin \pi_l(Y_l\cap T_l)$, we have $X_0\not\subseteq Y_l$. Similarly, since $\alpha_{l+1}\in \pi_i(X_0\cap T_i)$ and $\alpha_{l+1}=\alpha_i\notin \pi_i(Y_i\cap T_i)$, we have $X_0\not\subseteq Y_i$. We now show that $Y_l$ is $\pi_i$-tight. For each $0\leq j<l$ with $u_j\in Y_l$, we have $Y_j=Y_l$ by Claim \ref{clm:y_i=y_j}, which implies that $\alpha_{j+1}\notin \pi_j(Y_l\cap T_j)$. Hence, we have $f_{\pi_l}(Y_l)\geq f_{\pi_i}(Y_l)$. Since $Y_l$ is $\pi_l$-tight, this implies that $Y_l$ is also $\pi_i$-tight. Then, $Y_l,X_0,Y_i$ are distinct sets satisfying $x_0\in Y_l\cap X_0\cap Y_i$, which implies by Claim \ref{clm:tight} that $Y_l\cup X_0$ is a $\pi_l$-tight set in $\mathcal{F}$, or $Y_l\cup Y_i$ is a $\pi_i$-tight set in $\mathcal{F}$, contradicting the maximality of $Y_l$ in $\mathcal{F}_l$, or that of $Y_i$ in $\mathcal{F}_i$.

Consider the case when $1\leq s\leq p$. If $X_s=X_{s-1}$, then we have $\alpha_{l+1},\beta\in \pi_l^{s-1}(X_s\cap T_l)=\pi_l^{s-1}(X_{s-1}\cap T_l)$, which contradicts that $X_{s-1}$ is not $\pi_l^{s-1}$-satisfying. Hence, we may assume that $X_s\neq X_{s-1}$. Since $\alpha_{l+1}\in \pi_i(X_s\cap T_i)\cap \pi_i(X_{s-1}\cap T_i)$ and $\alpha_{l+1}\notin \pi_i(Y_i\cap T_i)$, we have $X_s,X_{s-1}\not\subseteq Y_i$. If $X_s$ and $X_{s-1}$ are $\pi_i$-tight, then since $x_s\in Y_i\cap X_s\cap X_{s-1}$, by Claim \ref{clm:tight} one of $Y_i\cup X_s$ and $Y_i\cup X_{s-1}$ is a $\pi_i$-tight set in $\mathcal{F}$, contradicting the maximality of $Y_i$. Otherwise, at least one of $X_s$ and $X_{s-1}$ is not $\pi_i$-tight. Assume that $X_m$ is not $\pi_i$-tight for some $m\in \{s,s-1\}$. Since $X_m$ is $\pi_l$-tight and not $\pi_i$-tight, $u_r\in X_m$ holds for some $i\leq r\leq l-1$. Since $\beta\in \pi_l(X_m\cap T_l)$ and $\beta\notin \pi_l(Z\cap T_l)$, we have $X_m\not\subseteq Z$. If $X_m\neq Y_r$, then since $u_r\in Z\cap X_m\cap Y_r$, by Claim \ref{clm:tight} $Z\cup X_m$ is a $\pi_l$-tight set in $\mathcal{F}$, or $Z\cup Y_r$ is a $\pi_r$-tight set in $\mathcal{F}$, contradicting the maximality of $Z$. Hence, we have $X_m=Y_r$. Here, we have $f_{\pi_r}(Y_r)\geq f_{\pi_i}(Y_r)$, which implies that $Y_r$ is $\pi_i$-tight because $Y_r$ is $\pi_r$-tight. However, this contradicts that $X_m$ is not $\pi_i$-tight.
\end{proof}
By Claim \ref{clm:final}, every $X\in \mathcal{F}$ is $\pi_q^{p+1}$-satisfying, which contradicts the maximality of $T_0$.
\end{proof}

\section{Implication of the supermodular extension}
\label{sec:impl}
In this section, we prove that Theorem \ref{thm:vizing_supermodular} includes Theorems \ref{new-edge-coloringthm} and \ref{cor:supermo} as special cases. To show that Theorem \ref{thm:vizing_supermodular} includes Theorem \ref{new-edge-coloringthm}, suppose that $c(v)\leq \min\{\mathrm{deg}(v),k\}$ holds for every $v\in V$, and $S=\{v\in V\mid c(v)+\mu(v)>k\}$ is a stable set. Let $\mathcal{F}=\{\delta(v)\mid v\in V\} \subseteq 2^E$. Then, $\mathcal{F}$ is an intersecting 2/3-laminar family because $\delta(v_1)\cap \delta(v_2)\cap \delta(v_3)=\emptyset$ for every distinct vertices $v_1,v_2,v_3\in V$. Let $g:\mathcal{F}\rightarrow \mathbf{Z}$ be a function defined as follows:
\begin{align}
    g(X)=\max\{c(v)\mid v\in V,\ X=\delta(v)\}\ \ (X\in \mathcal{F}).\notag
\end{align}
Since $\delta(v_1)\cap \delta(v_2)\cap \delta(v_3)=\emptyset$ for every distinct vertices $v_1,v_2,v_3\in V$, $g$ is an intersecting 2/3-supermodular function. Let $\mathcal{L}=\{X\in \mathcal{F}\mid g(X)+D_{\mathcal{F}}(X)>k\}$.
Then $\mathcal{L}$ is a $g$-laminar family because we have $c(v)+\mu(v)\geq g(\delta(v))+D_{\mathcal{F}}(\delta(v))$ for every $v\in V$ with $g(\delta(v))=c(v)$, which implies that $S_{\mathcal{L}}=\{v\in V\mid \delta(v)\in \mathcal{L},\ g(\delta(v))=c(v)\}$ satisfies $S_{\mathcal{L}}\subseteq S$, and hence $S_{\mathcal{L}}$ is a stable set, which concludes that every distinct $X,Y\in \mathcal{L}$ satisfy $X\cap Y=\emptyset$. We also have
\begin{align}
    \min\{|\delta(v)|,k\}=\min\{\mathrm{deg}(v),k\}\geq c(v)=g(\delta(v))\notag
\end{align}
for every $v\in V$ with $g(\delta(v))=c(v)$. Therefore, by Theorem \ref{thm:vizing_supermodular}, there exists an assignment of colors $\pi:E\rightarrow [k]$ such that $|\pi(\delta(v))|\geq g(\delta(v))\geq c(v)$ holds for every $v\in V$, which implies Theorem \ref{new-edge-coloringthm}.

Theorem \ref{cor:supermo} is also a special case of Theorem \ref{thm:vizing_supermodular} as follows. Let $\mathcal{F}$ be an intersecting family, and $g:\mathcal{F}\rightarrow \mathbf{Z}$ an intersecting supermodular function. Then $\mathcal{F}$ is an intersecting 2/3-laminar family, and $g$ is an intersecting 2/3-supermodular function. Moreover, $\mathcal{L}=\{X\in \mathcal{F}\mid g(X)+D_{\mathcal{F}}(X)>k\}$ is a $g$-laminar family because if $X,Y\in \mathcal{L}$ satisfy $X\cap Y\neq \emptyset$, then we have $X\cup Y,X\cap Y\in \mathcal{F}$ and $g(X)+g(Y)\leq g(X\cup Y)+g(X\cap Y)$. Hence, by Theorem \ref{thm:vizing_supermodular}, if $\min\{|X|,k\}\geq g(X)$ holds for every $X\in \mathcal{F}$, then there exists an assignment of colors $\pi:U\rightarrow [k]$ such that $|\pi(X)|\geq g(X)$ holds for every $X\in \mathcal{F}$, which implies Theorem \ref{cor:supermo}.

\section{Polynomial time algorithm}
\label{sec:polyalgo}
In this section, we prove Theorem \ref{thm:polytime} from the constructive proof of Theorem \ref{thm:vizing_supermodular} in Section \ref{chap:supermodular_extension} with the aid of Theorem \ref{thm:min23}. To construct a coloring of Theorem \ref{thm:vizing_supermodular} with $\mathcal{F}=2^U$, we start with the empty coloring $\pi:\emptyset\rightarrow [k]$. Suppose that $\pi_0:T_0\rightarrow [k]$ is a current coloring such that 
\begin{align}
\label{eq:coloring_condition}
|X\setminus T_0|+|\pi_0(X\cap T_0)|\geq g(X)
\end{align}
holds for every $X\subseteq U$, where $T_0$ is a subset of $U$. Then, as in the proof of Theorem \ref{thm:vizing_supermodular}, we update the coloring $\pi_0$ to another coloring  $\pi:T\rightarrow [k]$ satisfying (\ref{eq:coloring_condition}) for every $X\subseteq U$, where $T=T_0\cup \{u_0\}$ for some $u_0\in U\setminus T_0$. By repeating this procedure of updates, we finally obtain a coloring of $U$ satisfying (\ref{eq:coloring_condition}) for every $X\subseteq U$, which is a desired coloring in Theorem \ref{thm:vizing_supermodular}. Hence, to prove Theorem \ref{thm:polytime}, it suffices to show that the update in the proof of Theorem \ref{thm:vizing_supermodular} can be done in polynomial time. 

The update starts with taking some element $u_0\in U\setminus T_0$. For each element $u\in U\setminus \{u_0\}$, let $\mathcal{F}_{u,u_0}=\{X\subseteq U\mid u,u_0\in X\}$. To obtain maximal sets in $\mathcal{F}_0$, we compute a minimizer of $f_{u,u_0}:\mathcal{F}_{u,u_0}\rightarrow \mathbf{Z}$ defined as follows for each $u\ne u_0$:
\begin{align}
f_{u,u_0}(X)=M(|X\setminus T_0|+|\pi_0(X\cap T_0)|-g(X))-|X|\ \ \ (X\in \mathcal{F}_{u,u_0}),\notag
\end{align}
where $M$ is an integer with $M>|U|$. If $\mathcal{F}_{u,u_0}\cap \mathcal{F}_0\ne \emptyset$, then the minimum value of $f_{u,u_0}$ is negative, and $X\in \mathcal{F}_{u,u_0}\cap \mathcal{F}_0$ with maximum cardinality minimizes $f_{u,u_0}$. If $\mathcal{F}_{u,u_0}\cap \mathcal{F}_0= \emptyset$, then the minimum value of $f_{u,u_0}$ is positive.
Recall that the number of maximal sets in $\mathcal{F}_0$ is at most two as shown in the proof of Theorem \ref{thm:vizing_supermodular}. If $X\ne \{u_0\}$ is the unique maximal set in $\mathcal{F}_0$, then $X$ is the unique minimizer of $f_{u,u_0}$ for each $u\in X\setminus \{u_0\}$. If $X_1$ and $X_2$ are the maximal sets in $\mathcal{F}_0$, then $X_i$ is the unique minimizer of $f_{u_i,u_0}$ for $i=1,2,$ where $u_1\in X_1\setminus X_2$ and $u_2\in X_2\setminus X_1$. Hence, we can compute all the maximal sets in $\mathcal{F}_0$ by minimizing $f_{u,u_0}$ for each $u\ne u_0$. Since $f_{\pi_0}(X)=|X\setminus T_0|+|\pi_0(X\cap T_0)|$ satisfies the submodular inequality by Claim \ref{clm:submo}, $f_{u,u_0}$ is 2/3-submodular if we regard $\mathcal{F}_{u,u_0}$ as $2^{U\setminus \{u,u_0\}}$. Therefore, a minimizer of $f_{u,u_0}$ can be obtained in polynomial time by Theorem \ref{thm:min23}. The update next constructs a maximal sequence $\{(Y_0,u_0),\ldots,(Y_l,u_l)\}$. For this, we need to compute a maximal set $Y_i$ in $\mathcal{F}_i$ distinct from $Z$ for each $i=1,\ldots,l$. This can be done in polynomial time by a similar way as the case when there are two maximal sets in $\mathcal{F}_0$. Similarly, we also need to obtain a maximal set in $\mathcal{F}_{l+1}$ distinct from $Z$, which can be done in polynomial time by the same way. After that, the update proceeds to construct a maximal sequence $\{x_0,\ldots,x_p\}$. To
obtain this, we need to find a set $X_i\subseteq U$ that is not $\pi_l^i$-satisfying for each $i=0,\ldots,p-1$, and verify that $X_i$ is the unique one. Since the number of sets that are not $\pi_l^p$-satisfying is at most one by Claim \ref{clm:notsat1}, if we find a set $X_i$ that is not $\pi_l^i$-satisfying, then we can see that $X_i$ is the unique one. Hence, the verification part is unnecessary. Let $f_i:2^U\rightarrow \mathbf{Z}$ be a set function defined as follows for each $i=0,\ldots,p-1$:
\begin{align}
f_i(X)=|X\setminus T_l|+|\pi_l^i(X\cap T_l)|-g(X)\ \ \ (X\subseteq U).\notag
\end{align}
If there exists a set $X_i\subseteq U$ that is not $\pi_l^i$-satisfying, then the minimum value of $f_i$ is negative, and $X_i$ is the unique minimizer of $f_i$. Otherwise, the minimum value of $f_i$ is nonnegative. Hence, we can compute the unique set $X_i\subseteq U$ that is not $\pi_l^i$-satisfying by minimizing $f_i$ for each $i=0,\ldots,p-1$. Let $\mathcal{F}_u=\{X\subseteq U\mid u\in X\}$, and let $f_{i,u}$ be the restriction of $f_i$ to $\mathcal{F}_u$ for each $u\in U$. Then $f_{i,u}$ is 2/3-submodular if we regard $\mathcal{F}_u$ as $2^{U\setminus \{u\}}$. Hence, we can obtain a minimizer of $f_i$ in polynomial time by minimizing $f_{i,u}$ for each $u\in U$. In addition, we also need to compute the unique set $X_p\subseteq U$ that is not $\pi_l^p$-satisfying. This can be done in polynomial time by a similar way. The other parts of the update can easily be done in polynomial time.

\section*{Acknowledgments}
The author is grateful to Krist\'{o}f B{\'{e}}rczi and Tam\'{a}s Schwarcz for helpful discussions on 2/3-submodular functions. This work was supported by Grant-in-Aid for JSPS Fellows Grant Number JP23KJ0379 and JST SPRING Grant Number JPMJSP2108.

\bibliography{main}

\begin{thebibliography}{10}

\bibitem{barasz2006}
M.~B{\'{a}}r{\'{a}}sz.
\newblock Matroid intersection for the min-rank oracle.
\newblock {\em Technical Report QP-2006-03, Egerv{\'{a}}ry Research Group},
  2006.

\bibitem{berczi2008}
K.~B{\'{e}}rczi and A.~Frank.
\newblock Variations for {Lov{\'{a}}sz'} submodular ideas.
\newblock In {\em Building Bridges}, pages 137--164. Springer Berlin
  Heidelberg, 2008.

\bibitem{berczi2023}
K.~B{\'{e}}rczi, T.~Kir{\'{a}}ly, Y.~Yamaguchi, and Y.~Yokoi.
\newblock Matroid intersection under restricted oracles.
\newblock {\em {SIAM} Journal on Discrete Mathematics}, 37:1311--1330, 2023.

\bibitem{fournier1991}
{C. Berge and J.-C. Fournier}.
\newblock A short proof for a generalization of {V}izing's theorem.
\newblock {\em Journal of Graph Theory}, 15:333--336, 1991.

\bibitem{fournier1977}
J.-C. Fournier.
\newblock M\'ethode et th\'eor\`eme g\'en\'eral de coloration des ar\^etes d'un
  multigraphe.
\newblock {\em Journal de Math\'ematiques pures et appliqu\'ees}, 56:437--453,
  1977.

\bibitem{frank2014}
A.~Frank, T.~Kir{\'a}ly, J.~Pap, and D.~Pritchard.
\newblock Characterizing and recognizing generalized polymatroids.
\newblock {\em Mathematical Programming}, 146:245--273, 2014.

\bibitem{gupta1974}
R.~P. Gupta.
\newblock On decompositions of a multi-graph into spanning subgraphs.
\newblock {\em Bulletin of the American Mathematical Society}, 80:500--502,
  1974.

\bibitem{gupta1978}
R.~P. Gupta.
\newblock An edge-coloration theorem for bipartite graphs with applications.
\newblock {\em Discrete Mathematics}, 23:229--233, 1978.

\bibitem{iwata2018}
S.~Iwata and Y.~Yokoi.
\newblock List supermodular coloring.
\newblock {\em Combinatorica}, 38:1437--1456, 2018.

\bibitem{konig1916}
D.~K\H{o}nig.
\newblock Graphok \'{e}s alkalmaz\'{a}suk a determin\'{a}nsok \'{e}s a halmazok
  elm\'{e}let\'{e}re.
\newblock {\em Mathematikai \'{e}s Term\'{e}szettudom\'{a}nyi
  \'{E}rtesit\H{o}}, 34:104--119, 1916.

\bibitem{mizutani}
R.~Mizutani and Y.~Yoshida.
\newblock Polynomial algorithms to minimize 2/3-submodular functions.
\newblock In {\em 25th Conference on Integer Programming and Combinatorial
  Optimization (IPCO 2024)}, to appear.

\bibitem{schrijver1985}
A.~Schrijver.
\newblock Supermodular colourings.
\newblock In L.~Lov\'{a}sz and A.~Recski, editors, {\em Matroid Theory}, pages
  327--343, North-Holland, 1985.

\bibitem{tardos1985}
\'{E}. Tardos.
\newblock Generalized matroids and supermodular colourings.
\newblock In L.~Lov\'{a}sz and A.~Recski, editors, {\em Matroid Theory}, pages
  359--382, North-Holland, 1985.

\bibitem{vizing1965}
V.~G. Vizing.
\newblock The chromatic class of a multigraph.
\newblock {\em Cybernetics}, 1:32--41, 1965.

\bibitem{yokoi2019}
Y.~Yokoi.
\newblock List supermodular coloring with shorter lists.
\newblock {\em Combinatorica}, 39:459--475, 2019.

\end{thebibliography}
\bibliographystyle{plain}
\end{document}